\documentclass[3p]{elsarticle}

\usepackage{amssymb}
\usepackage{amsmath}

\usepackage{multirow}
\usepackage{enumerate}
\usepackage[colorlinks,
            linkcolor=blue,
            anchorcolor=blue,
            citecolor=blue
            ]{hyperref}

\usepackage{enumitem}  
\usepackage{adjustbox}
\usepackage{caption}
\newcommand{\rmnum}[1]{\romannumeral #1}
\newcommand{\Rmnum}[1]{\expandafter\@slowromancap\romannumeral #1@}
\makeatother

\newtheorem{theorem}{Theorem}
\newtheorem{lemma}{Lemma}
\newtheorem{proposition}{Proposition}
\newdefinition{remark}{Remark}
\newproof{proof}{Proof}

\newtheorem{lemmaA}{Lemma}

\journal{Applied Mathematics and Computation}

\begin{document}

\begin{frontmatter}

\title{Exponential Runge-Kutta Galerkin finite element method for a reaction-diffusion system with nonsmooth initial data}

\author{Runjie Zhang} \ead{2112314030@mail2.gdut.edu.cn}
\author{Shuo Yang} \ead{2112414027@mail2.gdut.edu.cn}
\author{Jinwei Fang\corref{cor}} \ead{fangjinwei@gdut.edu.cn} \cortext[cor]{Corresponding author}

\address{School of  Mathematics and Statistics, Guangdong University of Technology, Guangzhou, Guangdong 510006, PR China}

\begin{abstract}
This study presents a numerical analysis of the Field–Noyes reaction–diffusion model with nonsmooth initial data, employing a linear Galerkin finite element method for spatial discretization and a second-order exponential Runge–Kutta scheme for temporal integration. The initial data are assumed to reside in the fractional Sobolev space \(H^{\gamma}(\Omega)\) with \(0 < \gamma < 2\), where classical regularity conditions are violated, necessitating specialized error analysis.  By integrating semigroup techniques and fractional Sobolev space theory, sharp fully discrete error estimates are derived in both \(L^2\) and \(H^1\) norms. This demonstrates that the convergence order adapts to the smoothness of initial data, a key advancement over traditional approaches that assume higher regularity. Numerical examples are provided to support the theoretical analysis. 
\end{abstract}


\end{frontmatter}

\section{Introduction}

Reaction-diffusion systems are widely employed to characterize spatiotemporal pattern formation in chemical and biological systems. A paradigmatic example is the Belousov–Zhabotinsky (BZ) reaction, which involves over ten elementary chemical reactions occurring concurrently and exhibits a persistent deviation from chemical equilibrium. This system is recognized as a canonical manifestation of self-organization phenomena in chemistry. To mathematically formalize this complex behavior, Field and Noyes proposed a simplified mathematical model from a macroscopic perspective in their seminal work \cite{FN74b}. In the present study, we investigate the initial-boundary value problem for the Field–Noyes model:
\begin{equation}\label{eq:eq1}
    \left\{\;
        \begin{aligned} 
        &\frac{\partial u_1}{\partial t}=a_1\Delta u_1+\lambda^{-1}(\rho u_3-u_1u_3+u_1-u_1^2),\quad &&\text{in}\,\,\Omega\times(0,T],\\
        &\frac{\partial u_2}{\partial t}=a_2\Delta u_2+u_1-u_2,\quad &&\text{in}\,\,\Omega\times(0,T],\\
        &\frac{\partial u_3}{\partial t}=a_3\Delta u_3+\delta^{-1}(-\rho u_3-u_1 u_3+cu_2),\quad &&\text{in}\,\,\Omega\times(0,T],\\
        &\frac{\partial u_1}{\partial n(x)}=\frac{\partial u_2}{\partial n(x)}=\frac{\partial u_3}{\partial n(x)}=0,\quad &&\text{on}\,\,\partial\Omega\times(0,T],\\
        &u_1(x,0)=u_{1,0}(x),\,u_2(x,0)=u_{2,0}(x),\,u_3(x,0)=u_{3,0}(x),\quad &&\text{in}\,\,\Omega,
        \end{aligned}
    \right.
\end{equation}
where $\Omega\in \mathbb{R}^d$, $d = 1, 2, 3$ is a convex polygonal domain with boundary $\partial\Omega$. 
Here, $u_1(x,t)$, $u_2(x,t)$, $u_3(x,t)$ denote the concentration of chemical substances in the domain $\Omega$. The parameters $\delta$, $\lambda$, $\rho$ and $c$ represent positive constants associated with chemical properties, while $a_1$, $a_2$, $a_3$ signify positive diffusion coefficients. The well-posedness of this problem was investigated in \cite[Chapter 10]{ABsPEE}, and relevant results will be discussed in Section \ref{sec:pre}. This paper is devoted to the numerical analysis of this semilinear parabolic problem with \emph{nonsmooth initial data}.

In recent decades, exponential integrators have emerged as highly efficient schemes for the temporal discretization of parabolic equations. These methods effectively address stiffness by incorporating precise treatment of stiff terms within the numerical framework. Driven by advancements in computational algorithms, substantial research efforts have been devoted to the development of exponential integrators for semilinear parabolic problems. A variety of such integrators have been explored, including exponential Runge–Kutta methods \cite{ERK1,ERK2,luan14stiff}, exponential multistep schemes \cite{EMult1,EMult2}, and exponential Rosenbrock methods \cite{ERB1,ERB2}, among others. Comprehensive overviews of exponential integration techniques can be found in \cite{EI_hochbruck,EI_review}. This study utilizes a second-order exponential Runge–Kutta method for temporal discretization.

To facilitate the error estimation for discrete methods, it is customary to impose certain boundedness assumptions on the derivatives of the solution and the nonlinear term of the semilinear parabolic problem. Such assumptions hold when the initial data and nonlinear term are sufficiently smooth and satisfy appropriate compatibility conditions.
However, these boundedness properties are not satisfied for nonsmooth initial data, rendering conventional error estimates inapplicable. This necessitates a specialized error analysis for nonsmooth cases.

Considerable research has been devoted to the nonsmooth error analysis of \emph{abstract} semilinear parabolic equations, encompassing both spatial and temporal discretization.
For spatial discretization, error estimates for Galerkin approximations are well-documented, with comprehensive results synthesized in the monograph \cite{GaleFEM}. Concerning temporal discretization, diverse schemes have been investigated, including fully-implicit, semi-implicit, exponential Rosenbrock, and implicit-explicit methods \cite{IRK1,IRK3,IRK2,mukam,24IMEX}. These studies have demonstrated the phenomenon of order reduction in time-discrete methods under nonsmooth initial conditions.

Furthermore, rigorous nonsmooth error analysis has been developed for \emph{specific} problems. For the Navier–Stokes equations, a substantial body of research exists, as synthesized in \cite{21NSH1}. For instance, \cite{08He} proved first-order convergence for the Euler implicit/explicit scheme, while \cite{10He} established suboptimal $1.5$-order convergence using a Crank–Nicolson/Adams–Bashforth scheme, both requiring $H^1$ initial data. Under $L^2$ initial conditions, first-order convergence of a variable-stepsize semi-implicit method was demonstrated in \cite{22Li_L2}. Additionally, \cite{18_Burges} extended the analysis of the Crank–Nicolson/Adams–Bashforth scheme to the Burgers equation, establishing $1.5$-order convergence under $H^1$ initial conditions. These error estimates are derived via energy method techniques.
Herein, distinct methodologies are employed for the reaction-diffusion equation (\ref{eq:eq1}) to extend these investigations and derive a broader class of error analysis results.

In this article, a linear finite element discretization with mesh size $h$ is employed for spatial approximation, and a second-order exponential Runge–Kutta scheme with time step $\Delta t$ is adopted for temporal discretization  of the Field–Noyes model (\ref{eq:eq1}). The initial data are assumed to reside in the fractional Sobolev space $H^\gamma(\Omega)$, where $0 < \gamma < 2$. This assumption on the fractional Sobolev regularity of the initial conditions is imposed to precisely characterize their smoothness properties. By employing semigroup techniques in the error analysis, the following error bounds for the fully discrete scheme are derived:  
\begin{equation}\label{full_err}
	\begin{aligned}
		\| u(t_n) - u^h_n \|_{L^2(\Omega)} &\leq C t_n^{-1+\gamma/2} h^2 + C t_n^{q/2} \Delta t^{\min(1+\gamma-\varepsilon,\,2)}, \quad 0 < t_n \leq T, \\
		\| u(t_n) - u^h_n \|_{H^1(\Omega)} &\leq C t_n^{-1+\gamma/2} h + C t_n^{q/2-1/2} \Delta t^{\min(1+\gamma-\varepsilon,\,2)}, \quad 0 < t_n \leq T,
	\end{aligned}
\end{equation}  
where $u = (u_1, u_2, u_3)$, $u^h_n$ denotes the numerical solution at time $t = t_n$, $q$ is defined in \eqref{def_q}, and $\varepsilon$ is an arbitrarily small positive number. These estimates are shown to be sharp, as validated by the numerical experiments in Section \ref{sec:experiments}. Moreover, as demonstrated in \cite{EMSV_high_order,21NSH1}, variable time-stepping strategies can mitigate order reduction phenomena, and sharper error bounds permit coarser time grids without order degradation.  

The main contributions of this work, in contrast to prior investigations,  are outlined as follows:  

\begin{enumerate}
	\item  Theoretical frameworks in existing literature \cite{IRK1,IRK3,IRK2,24IMEX,mukam} impose relatively weak assumptions on equation nonlinearity, which are insufficient to derive sharp convergence rates for problem \eqref{eq:eq1}. Notably, these frameworks fail to encompass certain low-regularity scenarios. By contrast, our approach directly addresses the nonlinear term of \eqref{eq:eq1}, yielding sharp error bounds that hold for a broader class of initial data distributions.
	
	\item Whereas previous investigations \cite{22Li_L2,10He,08He} focused on temporal discretization errors for equations with initial data in integer-order Sobolev spaces ($L^2$, $H^1$, and $H^2$), the present study analyses the reaction–diffusion system (\ref{eq:eq1}) with initial data in the fractional Sobolev space $H^{\gamma}$ ($0 < \gamma < 2$). A key finding is the demonstration that the convergence order (\ref{full_err}) exhibits continuous dependence on the regularity parameter $\gamma$, establishing a novel connection between solution smoothness and numerical accuracy.
	
	\item The analysis of the three-dimensional case, when combined with the fractional regularity of initial data, presents a primary technical challenge. To address this, we develop novel theoretical tools for nonlinear term treatment and refine analytical techniques, enabling rigorous error characterization in regimes where prior methods were inapplicable.
	
\end{enumerate}

The paper is structured as follows: Section \ref{sec:pre} introduces the notational conventions, presents fundamental results for the Field–Noyes model, and establishes a key technical lemma that underlies subsequent analyses. Subsequently, Section \ref{sec:spadis_erran} develops the spatial discretization framework, where critical estimates \eqref{localLips_discrete}–\eqref{AhDuufh} are derived by leveraging the specific structure of the nonlinear term. This section further establishes spatial discretisation errors and associated estimates for the solution and nonlinear term of \eqref{eq2}, providing the foundational analytical tools for the subsequent temporal analysis. Section \ref{sec:fuldis_erran} then focuses on the temporal discretisation error analysis for the semidiscrete scheme \eqref{eq2}, building upon the spatial estimates to derive comprehensive error bounds. Finally, Section \ref{sec:experiments} presents numerical experiments that validate the theoretical findings, with concluding remarks summarised in Section \ref{sec:conclusions}.

\section{Preliminaries}
\label{sec:pre}
We begin by introducing some notations. 
For $s\geq 0$, we denote by $\|\cdot\|_s$ the norm of the Sobolev spaces $H^s=H^s(\Omega)$ over the domain $\Omega$ (see, e.g., \cite{ABsPEE, multiplication21}). For $s=0$, we identify  $H^0(\Omega)$ with $L^2(\Omega)$ equipped with the norm  $\|\cdot\|$ and inner product $(\cdot,\cdot)$. The $L^\infty$ space consists of all bounded measurable functions in $\Omega$. 
For $s\geq 1$, the space $H_N^s=H_N^s(\Omega)$ denotes the Sobolev space $H^s$ subject to homogeneous Neumann boundary conditions.  
For convenience, we denote by $C$ a generic positive constant and by $\varepsilon$ a sufficiently small positive number, both of which may vary across instances.

\subsection{Well-posedness of the problem}
This subsection presents existence, uniqueness and regularity results for the Field-Noyes model.   Further details can be found in \cite[Chapter 10]{ABsPEE}.

Let $X=(L^2)^3$ be the underlying space. 
The space $\mathcal{L}(X)$ denote the set of bounded linear operators from $X$ to $X$ with operator norm $\Vert\cdot\Vert_{\mathcal{L}(X)}$. For simplicity, $\|\cdot\|_X$, $(u,v)_X$ and $\|\cdot\|_{(H^s)^3}$ are denoted as $\|\cdot\|$, $(\cdot,\cdot)$ and $\|\cdot\|_s$ respectively. 
Next, we introduce the linear operator 
\begin{equation*}
	A=\text{diag}\{A_1,A_2,A_3\},
\end{equation*}
where $A_1,A_2$ and $A_3$ are realizations of $a_1\Delta-\lambda^{-1}I$, $a_2\Delta-I$ and $a_3\Delta-\delta^{-1}\rho\,I$, respectively, in $L^2$ under the homogeneous Neumann boundary conditions on $\partial\Omega$. 
Consequently, $A$ is a negative definite self-adjoint operator on $X$ and serves as the infinitesimal generator of the analytic semigroup $S(t)=\text{e}^{tA}$ on $X$. The domains of its fractional powers are specified as
\begin{equation*}
    \left\{\;
    \begin{aligned}
        &D(A^\theta)=({H}^{2\theta})^3,\qquad \text{if}\:\: 0\leq\theta<\frac{3}{4},\\
        &D(A^\theta)=({H}_N^{2\theta})^3,\qquad \text{if}\:\:\frac{3}{4}<\theta\leq 1,
    \end{aligned}\right.
\end{equation*}
with norm equivalence 
\begin{align}
    C^{-1}\Vert u\Vert_{2\theta}\leq\Vert A^\theta u\Vert \leq C\Vert u\Vert _{2\theta}, \:\:u\in D(A^\theta).\label{ONEH}
\end{align}
Subsequently, we define the nonlinear operator $f:D(A^\theta)\to X$ as follows:
\begin{equation}
    f(u)=
    \begin{pmatrix}
        \lambda^{-1}(\rho u_3-u_1u_3+2u_1-u_1^2)\\
        u_1\\
        \delta^{-1}(-u_1u_3+cu_2)
    \end{pmatrix},\label{def_f}
\end{equation}
where $\frac{d}{8}<\theta<1$. Additionally, $f$ can also be viewed as a mapping from $L^\infty\to L^\infty$. 

Then, we can reformulate the Field-Noyes model \eqref{eq:eq1} as an abstract evolution equation in $X$:
\begin{equation}
    \left\{\;
        \begin{aligned} 
        &\frac{du(t)}{dt}=Au(t)+f(u(t)), \quad t\in (0,T],\\
        &u(0)=u_0 \in D(A^{\gamma/2}).
        \end{aligned}
    \right.\label{eq2}
\end{equation}
This equation admits a unique global solution in the function space:
\begin{equation*}
	u\in \mathcal{C}((0,T];D(A))\cap \mathcal{C}([0,T];D(A^{\gamma/2}))\cap \mathcal{C}^1((0,T];X).
\end{equation*} 
It also holds that
\begin{align}
     \|A^{s/2}u(t)\|\leq C&t^{{\gamma}/2-s/2}+C, \quad  0\leq s\leq 2, \label{bound_u}\\
     \Vert f(u(t)) \Vert&\leq Ct^{-1+{\gamma}/2},\label{f_L2bou}
\end{align}
where $C$ depend on $u_0$, $\Omega$ and $T$. 
It is well-known that the solution of \eqref{eq2} can be expressed by the variation of constants formula
\begin{align}
    u(t)=S(t)u_0+\int_{0}^{t}S(t-\tau)f(u(\tau))\text{d}\tau,\qquad t\in[0,T].\label{VoCF1}
\end{align}

\subsection{Lipschitz condition for the nonlinear term}
First, we introduce the set 
\begin{equation}\label{Set_B}
    \begin{aligned}
        \mathcal{B}=&\left\{(s,s_1,s_2)\mid -\frac{3}{2}<s\leq s_i<\frac32\:\: \text{for}\:\: i=1,2,\:\:\:\right.  \\
        &\quad \left. s_1+s_2> \frac d2+s  ,\:\:\:s_1+s_2>0, \:\: s_1>0 \:\right\},
    \end{aligned}
\end{equation}
which will be frequently used. Here $s_1$ is positive and represent the boundedness of the solution. And it is desirable that $s_1$ be as large as possible. 
Then we present a important lemma that effectively controls the nonlinear term $f(u)$ and its Fréchet derivatives. 
\begin{lemma}\label{le:pointmult}
    Let $(s,s_1,s_2)\in \mathcal{B}$. Given $u\in D(A^{s_1/2})$, $v\in D(A^{s_2/2})\cap L^2$ and $uv\in L^2$, we have
  \begin{align*}
      \Vert A^{s/2}(u v)\Vert\leq C\Vert A^{s_1/2}u\Vert \Vert A^{s_2/2} v\Vert,
  \end{align*}
  where $C$ is independent of $u$ and $v$.
\end{lemma}

\begin{proof}
    We consider three cases separately:

    \textbf{Case 1:} $s\geq 0,s_2\geq 0$.
    Using \cite[Theorem 7.4]{multiplication21} and the equivalence of norms \eqref{ONEH}, the conclusion is evidently true.

    \textbf{Case 2:} $s<0,s_2<0$. For $w_1\in L^2$ and $\theta\geq 0$, we have
    \begin{align*}
        \sup_{w_2\in D(A^\theta)} \frac{|(w_1,w_2)|}{\Vert A^\theta w_2 \Vert}=\sup_{w_2\in D(A^\theta)} &\frac{|(A^{-\theta} w_1,A^{\theta} w_2)|}{\Vert A^\theta w_2 \Vert} \leq \Vert A^{-\theta} w_1 \Vert,
        \\
        \sup_{w_2\in D(A^\theta)} \frac{|(w_1,w_2)|}{\Vert A^\theta w_2 \Vert}\geq \frac{|(w_1,A^{-2\theta}w_1)|}{\Vert A^{-\theta} w_1 \Vert}&= \frac{|(A^{-\theta}w_1,A^{-\theta}w_1)|}{\Vert A^{-\theta} w_1 \Vert}\geq \Vert A^{-\theta} w_1 \Vert.
    \end{align*}
    Then $\sup_{w_2\in D(A^\theta)} \frac{|(w_1,w_2)|}{\Vert A^\theta w_2 \Vert} = \Vert A^{-\theta} w_1 \Vert$. 
    Similar to \cite[Theorem 8.1]{multiplication21}, it follows from a duality argument that
    \begin{align*}
        &\quad \Vert A^{s/2} (u v)\Vert = \sup_{w\in D(A^{-s/2})} \frac{|(u v,w)|}{\Vert A^{-s/2} w \Vert}=\sup_{w\in D(A^{-s/2})} \frac{|(v,uw)|}{\Vert A^{-s/2} w \Vert} \\
        &\leq \sup_{w\in D(A^{-s/2})} \frac{\Vert A^{s_2/2}v\Vert \Vert A^{-s_2/2}(u w)\Vert}{\Vert A^{-s/2} w \Vert}
        \leq \sup_{w\in D(A^{-s/2})} \frac{\Vert A^{s_2/2}v\Vert \Vert A^{s_1/2}u\Vert \Vert A^{-s/2}w\Vert }{\Vert A^{-s/2} w \Vert} 
        \\
        &  \leq \Vert A^{s_2/2}v\Vert \Vert A^{s_1/2}u\Vert. 
    \end{align*}
    where $(-s_2,s_1,-s)$ can be verified to be in $\mathcal{B}$.
    
    \textbf{Case 3:} $s<0,s_2\geq 0$. Employing the same techniques that were used in the proof of \cite[Theorem 8.2]{multiplication21}, let
    $$
        \dfrac{1}{r}=\max\biggl\{\dfrac{1}{2}-\dfrac{s_1}{d},\dfrac{1}{2}-\dfrac{s_2}{d},1-\dfrac{s_1+s_2}{d}+\varepsilon,\dfrac{1}{2}\biggr\}.$$
    According to \cite[Theorem 7.4]{multiplication21} and the equivalence of norms \eqref{ONEH}, we have $\Vert u v\Vert_{L^r}\leq \Vert u\Vert_{s_1}\Vert v\Vert_{s_2}\leq  \Vert A^{s_1/2}u\Vert\Vert A^{s_2/2}v\Vert $. 
    Now, it is enough to prove $\Vert A^{s/2} (u v)\Vert\leq \Vert u v\Vert_{L^r}$. Let $L^{r'}=(L^{r})^*$. By Sobolev embedding theorem, it follows that $D(A^{-s/2})\hookrightarrow L^{r'}$. 
    For $w\in D(A^{-s/2})\subset L^{r'} $, we get
    \begin{align*}
       \frac{|(u v,w)|}{\Vert A^{-s/2}w\Vert} \leq \frac{|(u v,w)|}{\Vert w\Vert_{L^{r'}}}.
    \end{align*}
    Taking the supremum on both sides, the conclusion can be drawn. \qed
\end{proof}

The $\frac{3}{2}$ constraint in the set $\mathcal{B}$ is imposed to facilitate the extension of Lemma \ref{le:pointmult} to its discrete analogue in next section. 
By the definition of $f$ in \eqref{def_f} and Lemma \ref{le:pointmult}, we obtain the following Lipschitz continuity of $f(u)$: 
\begin{equation*}
    \|A^{s/2}(f(u)-f(v))\|\leq C(\|A^{s_1/2}u\|+\|A^{s_1/2}v\|+1)\|A^{s_2/2}(u-v)\|.
\end{equation*}
where $(s,s_1,s_2)\in \mathcal{B}$, $u\in D(A^{s_1/2})\cap L^\infty$ and $v\in D(A^{s_2/2})\cap L^\infty$. 
This implies the bound of the nonlinearity
\begin{equation}
    \|A^{s/2}f(u)\|\leq \|f(u)-f(0)\|+\|f(0)\| \leq C(\|A^{s_1/2} u\|+1)\|A^{s_2/2}u\|. \label{BoNLT}
\end{equation}

\section{Error analysis for spatial discretization scheme}
\label{sec:spadis_erran} 

This section presents the spatial discretization of problem \eqref{eq2} via the Galerkin finite element method, leading to the approximate problem \eqref{eq3}. We first establish the relationship between the continuous and discrete problems, then analyze the spatial discretization errors. Finally, we provide key estimates for the solution $u^h(t)$ and the nonlinear term $f_h(u)$, which serve as the foundation for the temporal discretization error analysis in Section \ref{sec:fuldis_erran}.

\subsection{Spatial discretization scheme}
Let $\pi_{h}$ be the triangulation with a maximum element size of $h$ for the domain $\Omega$. We assume that $\pi_{h}$ is regular. Let $V_h$ be the space of  continuous and piecewise linear functions defined over the triangulation $\pi_h$ and $X_h=(V_h)^3$. The projection operator $P_h$ from $X$ to $X_h$ is defined as 
\begin{equation*}
    (P_h u,v_h)=(u,v_h),\quad \forall\: v_h\in X_h,\text{ for } u\in X.
\end{equation*}
For problem \eqref{eq2}, consider the sesquilinear form
\begin{equation*}
    \hat{a}_i(u,v)=-\int_{\Omega}a_i\nabla u \nabla v \text{d}x+\int_{\Omega}b_i u v \text{d}x,\qquad u,v\in H^1,
\end{equation*}
associated with $-A_i$, where $b_1=\lambda^{-1}$, $b_2=1$ and $b_3=\delta^{-1}\rho$. Define $a(u,v)=\hat{a}_1(u_1,v_1)+\hat{a}_2(u_1,v_2)+\hat{a}_3(u_1,v_2)$ where $u,v\in (H^1)^3$. 
Then the discrete operator $A_{h}:X_h\longrightarrow X_h$ is defined by
\begin{equation*}
    (-A_{h}u_h,v_h)={a}(u_h,v_h),\quad\forall\: v_h\in X_h,\text{ for }u_h\in X_h,\:i=1,2,3.
\end{equation*}
Subsequently, we define the Ritz operator $R_{h}:(H^1)^3\longrightarrow X_h$, 
\begin{equation*}
    {a}(R_{h}u,v_h)={a}(u,v_h),\qquad \forall\: v_h\in X_h,\text{ for }u\in (H^1)^3,\:i=1,2,3.
\end{equation*}
Combining these components, we obtain the Galerkin finite element scheme for \eqref{eq2}: find $u^h(t)\in X_h$, such that
\begin{equation}
	\left\{\;
	\begin{aligned} 
		&\frac{du^h(t)}{dt}=A_h u^h(t)+f_h(u^h(t)), \quad t\in (0,T],\\
		&u^h(0)=P_h u_0,
	\end{aligned}
	\right.
	\label{eq3}
\end{equation}
where $f_h(u^h(t))=P_h (f(u^h(t)))$. 

In this study, we utilize the semigroup approach for numerical analysis. It is observed that $-A$ and $-A_h$ are sectorial operators in $X$ and $X_h$, respectively, associated with the sesquilinear form $a(u,v)$. Thus, $A$ and $A_h$ serve as the infinitesimal generators of the analytic semigroup $S(t)=\text{e}^{tA}$ on $X$ and $S_h(t)=\text{e}^{tA_h}$ on $X_h$, respectively. The following estimates are valid for $S(t)$ and $S_h(t)$ (see \cite{SG_P,ABsPEE}).

\begin{lemma}\label{le:semigroup}
    Let $\alpha,\alpha'\in R$ and $0\leq \theta\leq 1$. Then the following estimates hold (see \cite{ABsPEE,SG_P}).
  \begin{align*}
      \Vert A^\alpha S(t)\Vert _{\mathcal{L}(X)} &\leq Ct^{-\alpha},& & t>0,\:\alpha\geq 0,& \\
      \Vert A^{-\theta}(I-S(t))\Vert _{\mathcal{L}(X)}&\leq Ct^\theta,& & t\geq 0,&\\
      A^\alpha S(t)&=S(t) A^\alpha, & &\text{on}\enspace D(A^\alpha),&\\
       D(A^\alpha)&\subset D(A^{\alpha'}), & &\text{if}\enspace \alpha\geq\alpha'.&
  \end{align*}
  These estimates hold with a uniform constant $C$ (independent of $h$) when $A$ and $S(t)$ are replaced by their discrete versions $A_h$ and $S_h(t)$  respectively.
\end{lemma}

We will establish the relationship between problem \eqref{eq2} and its spatial semidiscrete scheme \eqref{eq3}. 
\begin{lemma}\label{le:semidis_prop}
For spatial semi-discretization of problem \eqref{eq2}, the Galerkin finite element method \eqref{eq3} exhibits the following properties:
\begin{align}
    \|{u}_h\|_{s}&\leq C h^{-s}\|{u}_h\|, &&{u}_h \in X_h\text{ , } 0\leq s<3/2,& \label{uhs}  \\
    \Vert A^\theta u_h\Vert &\leq C\Vert A^\theta_h u_h\Vert,&& u_h\in X_h\text{, }-{3}/{4}\leq \theta < {3}/{4},&\label{Auh}   
    \\
    \Vert A_h^\theta P_h u\Vert &\leq C\|A^\theta u\|,&& u\in D(A^\theta)\cap L^2 \text{, }-{3}/{4}\leq \theta\leq 1,&\label{AhPh}\\\|(1-R_{h})u\|_s &\leq C h^{r-s}\|u\|_r, && u\in \:D(A^{r/2}),\:s\in[0,1],\:r\in[1,2].&\label{I_Rh}
\end{align}
\end{lemma}

\begin{proof}
    It is well known that the following error estimates hold (see e.g. \cite{TMToFEM,FEMoEP})
\begin{equation*}
	\|R_{h}v-v\| +h\|R_{h}v-v\|_1 \leq Ch^r\|v\|_r,\quad v\in D(A^{r/2}),\quad r=1,2.
\end{equation*}
Using operator interpolation theory \cite[Proposition 14.1.5]{TMToFEM}, we first interpolate the inequalities pairwise and then interpolate the resulting estimates to obtain the bound \eqref{I_Rh}.  

Additionally, for the case $d=2$ and $0\leq s<\frac32$, the space $X_h$ is continuously embedded in $(H^s)^3$ with the following estimate 
\begin{align*}
	\|{u}_h\|_{s}\leq C_s h^{-s}\|{u}_h\|,\qquad{u}_h \in X_h,
\end{align*}
see \cite[Proposition 7.1]{ABsPEE} for details. Analogous results hold for dimensions $d=1$ and $d=3$; the latter case is proved in Lemma \ref{lem:3dcase}. Based on this estimate, the bounds \eqref{Auh} and \eqref{AhPh} are obtained for the non-negative parameter range $\theta\geq 0$. 
For complete technical details, we refer to \cite{NY02} or \cite[Section 7.3]{ABsPEE}. The dual counterparts of \eqref{Auh} and \eqref{AhPh} can be established through the following inequality. 
\begin{align*}
    \Vert A_h^{-\theta}P_h u\Vert &= \sup_{v_h\in V_h} \frac{|(P_h u ,v_h)|}{\Vert A_h^{\theta} v_h \Vert}\leq C\sup_{v\in D(A^{\theta})} \frac{|(u,v)|}{\Vert A^{\theta} v \Vert} = C \Vert A^{-\theta} u \Vert,\\
    (A^{-\theta}u_h, v)&=a(u_h,R_h A^{-1-\theta}v)=(A_h^{-\theta} u_h,A_h^{1+\theta}R_h A^{-1-\theta}v)\\
    &=(A_h^{-\theta} u_h,A_h^{\theta}P_h A^{-\theta}v)\leq \Vert A_h^{-\theta} u_h\Vert \Vert v\Vert, 
\end{align*}
where $\theta>0$. This completes the proof. \qed
\end{proof}

\begin{remark}
	By combining \eqref{ONEH}, \eqref{Auh} and \eqref{AhPh}, we establish the following norm equivalence relation on the discrete space $X_h$: 
	\begin{align}
		\|A_h^\theta \cdot\| \sim \|&A^\theta \cdot\| \sim \|\cdot\|_{2\theta}, \quad  0 \leq \theta < \frac{3}{4}. \label{NormEq1}
	\end{align}
    This equivalence enables the application of semigroup methods for convergence analysis in fractional Sobolev spaces. The norm equivalence above and the semigroup properties in Lemma \ref{le:semigroup} are frequently used in subsequent analyses. To maintain conciseness, we will not explicitly reiterate them each time they are referenced.
\end{remark}

We observe that the pointwise multiplication $u_h v_h$ belongs to $(L^2)^3$ according to \cite[Theorem 7.4]{multiplication21} and the fact that  $X_h \subset (H^s)^3$ ($s<\frac{3}{2}$). Combining this result with the definition of $f$, Lemma \ref{le:pointmult}, and the bounds \eqref{Auh}-\eqref{AhPh}, we derive the key estimates for the nonlinearity $f_h(u_h)$: 
\begin{align}
    \|A_h^{s/2}(f_h(u_h)-f_h(v_h))\|
     &\leq C(\|A_h^{s_1/2}u_h\|+\|A_h^{s_1/2} v_h\|+1)\|A_h^{s_2/2} (u_h-v_h)\|,\label{localLips_discrete}\\
    \Vert A_h^{s/2} D_u f_h(u_h) v_h\Vert &\leq C(\Vert A_h^{s_1/2} u_h\Vert+1) \Vert A_h^{s_2/2} v_h\Vert,\label{AhDufh}
    \\
    \|A_h^{s/2} D_{uu} f_h(u_h)({v}_h, {w}_h)\| &\leq C\Vert A_h^{s_1/2}  {v}^h\Vert \Vert A_h^{s_2/2}  {w}_h\Vert,\label{AhDuufh}
\end{align}
where $(s,s_1,s_2)\in \mathcal{B}$, $u_h,v_h,w_h\in X_h$.

\subsection{Spatial error analysis}
This subsection is dedicated to the analysis of spatial discretization error. 
We start by investigating the error related to the linear part of problem \eqref{eq2}. 
\begin{proposition}\label{pro:lin_spa_err}
  Let $S(t)$ and $S_h(t)$ be the analytic semigroup generated by $A$ and $A_h$, respectively. For $s\in[0,1]$, $r\in [1,2]$, and $\alpha\in[0,r]$, if $w_0\in D(A^{\alpha/2})$, then the subsequent estimate holds
  \begin{align*}
      \|(S(t)-S_h(t)P_h) w_0\|_s\leq Ch^{r-s}t^{-(r-\alpha)/2}\|w_0\|_\alpha,\qquad t\in (0,T].
  \end{align*}
\end{proposition}

\begin{proof}
    The case $s = 0$ is established in \cite[Lemma 3.1]{LPSDE}. To derive the estimate for $s\in (0,1]$, we set
\begin{equation*}
    (S(t)-S_h(t)P_h) w_0=(S_h(t)P_h-R_hS(t))w_0+(R_h-I)S(t)w_0=:\theta(t)+\rho(t).
\end{equation*}
According to the proof in \cite[Lemma 3.1]{LPSDE}, we have 
\begin{equation*}
    \theta(t)=-S_h(t/2)P_h\rho(t/2)-\int_0^{t/2}A_h S_h(t-\tau)P_h\rho(\tau) \text{d}\tau-\int_{t/2}^t S_h(t-\tau)P_hD_s\rho(\tau) \text{d}\tau.
\end{equation*}
It follows from \eqref{AhPh} and semigroup properties that
\begin{align*}
    \Vert \theta(t)\Vert_s
    &\leq\left\Vert A_h^{s/2}S_h(t/2)P_h\rho(t/2)\right\Vert+\left\Vert\int_0^{t/2}A_h^{1+s/2} S_h(t-\tau)P_h\rho(\tau) \text{d}\tau \right\Vert\\
    &\quad\:\:+\left\Vert\int_{t/2}^t A_h^{s/2}S_h(t-\tau)P_hD_\tau\rho(\tau) \text{d}\tau \right\Vert
    \\
    &\leq C \Vert\rho(t/2)\Vert_s+C\int_0^{t/2}(t-\tau)^{-1}\Vert \rho(\tau)\Vert_s \text{d}\tau+\int_{t/2}^t \Vert D_\tau\rho(\tau)\Vert_s \text{d}\tau.
\end{align*}
Using the Ritz projection error estimate \eqref{I_Rh} yields
\begin{equation*}
    \left\{\;
    \begin{aligned}
        &\Vert\rho(t)\Vert_s \leq Ch^{r-s}\Vert A^{r/2}S(t)w_0\Vert\leq Ch^{r-s}t^{-(r-\alpha)/2} \Vert w_0 \Vert_\alpha, \\
        &\Vert D_t\rho(t)\Vert_1 \leq Ch^{r-s}\Vert A^{r/2+1}S(t)w_0\Vert\leq Ch^{r-s} t^{-1-(r-\alpha)/2}\Vert w_0 \Vert_\alpha.
    \end{aligned}\right.
\end{equation*}
Then we get
\begin{align*}
  \Vert \theta(t)\Vert_s &\leq Ch^{r-s} \Vert w_0\Vert_\alpha \left(t^{(\alpha-r)/2}+\int_0^{t/2}(t-\tau)^{-1}\Vert A^{(r-\alpha)/2} S(\tau)\Vert_{\mathcal{L}(X)} \text{d}\tau+\int_{t/2}^t \tau^{-1-(r-\alpha)/2}\text{d}\tau\right)
  \\
  &\leq Ch^{r-s} t^{-(r-\alpha)/2}\Vert w_0\Vert_\alpha,
\end{align*} 
where the treatment of the second term when $r=2$ and $\alpha=0$ follows from \cite[Lemma 2.1]{mukam}. This completes the proof. \qed
\end{proof}

We now proceed to analyze the spatial discretization error estimates. The analysis begins with establishing the well-posedness of the semidiscrete equation \eqref{eq3}. Applying \cite[Theorem 4.1]{ABsPEE}, we first obtain the local well-posedness of \eqref{eq3}. Building upon this result, we then analyze the spatial error for the local solution. 
In the subsequent analysis, we define $\hat{\gamma} = \min(3/2 - \varepsilon, \gamma)$, where $\varepsilon > 0$ is a small parameter ensuring relevant parameters lie in $\mathcal{B}$. Note that $\varepsilon$ may vary across different analytical contexts.

\begin{proposition}\label{pro:LSpa_err}
   Assuming there exist $h'>0$ and $T'>0$ such that for $h<h'$, equation \eqref{eq3} admits a solution $u^h(t)$ over the time interval $t\in[0,T']$. If $u^h(t)$ satisfies the estimate
  \begin{align}
      \|u^h(t)\|_{\hat{\gamma}}\leq C,\qquad 0\leq t\leq T',\label{uniform_error_condi}
  \end{align}
  then we have
  \begin{equation*}
      \|u(t)-u^h(t)\|_\mu \leq Ct^{-1+\gamma/2} h^{2-\mu},\qquad 0\leq \mu \leq 1,\:\: 0\leq t\leq T'.
  \end{equation*}
\end{proposition}

\begin{proof}
    Using the constant variation formula for $u^h(t)$ and subtracting \eqref{VoCF1}, we get
    \begin{equation}
        \begin{aligned}
            e(t)&:=\|u(t)-u^{h}(t)\|_\mu\\
            &\leq\|S(t)u_0-S_h(t)P_h u_0\|_\mu\\
            &\quad\:\:+\left\|\int_0^t S(t-\tau)f(u(\tau))\text{d}\tau-\int_0^tS_h(t-\tau)f_h(u^h(\tau))\text{d}\tau\right\|_\mu\\
            &=:e_1(t)+e_2(t),\label{SDE_e}
        \end{aligned}
    \end{equation}
    where $e_1(t)$ denoting the linear part and $e_2(t)$ the nonlinear part of the spatial error. Using Proposition \ref{pro:lin_spa_err} with $s=\mu$ and $r=2$, $\alpha=\gamma$ yields 
  \begin{equation}
      e_{1}(t)\leq Ct^{-1+\gamma/2} h^{2-\mu}\|u_{0}\|_{{\gamma}}.\label{SDE_e1}
  \end{equation}
  Estimates \eqref{bound_u} and \eqref{uniform_error_condi} imply the boundedness of $u(t)$ and $u^h(t)$, respectively. Select $r_1>0$ such that $(-r_1,\hat{\gamma},\mu)\in \mathcal{B}$ and $\mu+r_1<2$. It follows from \eqref{localLips_discrete} that
  \begin{align*}
    \Vert A_h^{-r_1/2}P_h(f(u(t))-f(u^h(t)))\Vert \leq C \Vert u(t)-u^h(t)\Vert_\mu.
  \end{align*}
  Select  $\varepsilon_1$ sufficiently small to ensure  that $(2\varepsilon_1,\hat{\gamma},d/2-2\varepsilon_1)\in \mathcal{B}$. From
  \eqref{BoNLT} and \eqref{bound_u}, we get
  $$
    \Vert f(u(t))\Vert_{2\varepsilon_1}\leq Ct^{\gamma/2-d/4+\varepsilon_1}+C. $$
  Using Proposition \ref{pro:lin_spa_err} with $s=\mu$, $r=2$, $\alpha=2\varepsilon_1$, combining the above estimates, we obtain
  \begin{equation}
      \begin{aligned}
          e_{2}(t)
          &\leq\int_{0}^{t}\left\|A_h^{\mu/2+r_1/2}S_h(t-\tau)A_h^{-r_1/2}P_h(f(u(\tau))-f(u^h(\tau)))\right\|\text{d}\tau \\
          &\quad\:\:+\int_0^t \| (S(t-\tau)-S_h(t-\tau)P_h)f(u(\tau))\|_\mu\text{d}\tau \\
          &\leq C\int_0^t (t-\tau)^{-\mu/2-r_1/2}e(\tau)\text{d}\tau+ Ch^{2-\mu}\int_{0}^{t}(t-\tau)^{-1+\varepsilon_1}\tau^{-1+\gamma/2}\text{d}\tau
          \\
          &\leq C\int_0^t (t-\tau)^{-\mu/2-r_1/2}e(\tau)\text{d}\tau+ C t^{-1+\gamma/2} h^{2-\mu}.\label{SDE_e2}
      \end{aligned}
  \end{equation}
  Substituting \eqref{SDE_e1} and \eqref{SDE_e2} into \eqref{SDE_e}, we have
  \begin{align*}
      e(t)\leq C t^{-1+\gamma/2}h^{2-\mu}+C\int_{0}^{t}(t-\tau)^{-\mu/2-r_1/2}e(\tau)\text{d}\tau.
  \end{align*}
  Applying an inequality of Gronwall type (see \cite[Theorem 1.27]{ABsPEE})  completes the proof. \qed
\end{proof}

This local error estimate serves as the foundation for extending the solution to the global domain. 
We now ready to establish the well-posedness of the approximate problem \eqref{eq3} and derive the corresponding global spatial error estimates.

\begin{theorem}\label{thm:space_err}
  Let $u(t)$ be the solution of \eqref{eq2}. For semidiscrete problem \eqref{eq3}, there exists $h'>0$ such that for $h<h'$, the unique solution $u^h(t)$  exists  on $[0,T]$. Moreover, the following estimates hold
  \begin{align}
    \Vert A_h^s u^h(t) \Vert \leq& Ct^{\gamma/2-s/2}+C, && 0\leq s\leq 2 \text{ , } 0 \leq t\leq T,\:&\label{bound_uh}
    \\
    \|u(t)-u^h(t)\|_{\mu}&\leq C t^{-1+\gamma/2}\: h^{2-\mu},&& 0\leq \mu\leq 1 \text{ , } 0 \leq t\leq T. &\label{spa_err}
  \end{align}
\end{theorem}

\begin{proof}
    Select a sufficiently small number $\varepsilon_1$ such that $2\varepsilon_1\leq \gamma$. 
    Using \eqref{localLips_discrete} with $(0,2\varepsilon_1,d/2-\varepsilon_1)\in\mathcal{B}$ yields the Lipschitz condition of the nonlinear term in discrete problem \eqref{eq3}.
    \begin{align*}  
        \|f_h(u_h)&-f_h(v_h)\|
        \leq C(\|A_h^{\varepsilon_1}u_h\|+\|A_h^{\varepsilon_1} v_h\|+1)\|A_h^{d/4-\varepsilon_1/2} (u_h-v_h)\|,
    \end{align*}
    where $u_h$, $v_h\in X_h$. 
    By \eqref{bound_u}, \eqref{AhPh} and \eqref{NormEq1}, there exist positive constants $M_1$, $M_2$ such that $\Vert A^{\varepsilon_1} u(t)\Vert + \Vert u(t)\Vert_{2\varepsilon_1} \leq M_1 $  for $t\in [0,T]$, $\Vert A_h^{\varepsilon_1} P_h u_0\Vert\leq M_2 \Vert A^{\varepsilon_1} u_0\Vert$ and $\Vert A_h^{\varepsilon_1} u_h\Vert\leq M_2 \Vert u_h\Vert_{2\varepsilon_1}$ for $u_h\in X_h$. 
    Let $M_2(M_1 + 1)$ be denoted as $M_{u_0}$. Then $\Vert A_h^{\varepsilon_1} P_h u_0\Vert\leq M_{u_0}$. 
    The existence and uniqueness of solutions $u^h(t)$ on $[0,\tau]$ can be established by applying \cite[Theorem 4.1]{ABsPEE} with $\beta=\varepsilon_1$, $\eta=\frac{d}{4}-\frac{\varepsilon_1}{2}$, where $\tau$ depends only on $M_{u_0}$. Furthermore, by employing \cite[Theorem 4.2]{ABsPEE} and noting the boundedness of $\Vert A_h^{\gamma/2}P_h u_0\Vert$, the estimates \eqref{bound_uh} and \eqref{spa_err} remain valid for $0\leq t\leq \tau$. 
    
  Subsequently, using Proposition \ref{pro:LSpa_err} with $\mu=2\varepsilon_1$, we have
  \begin{align*}
      \Vert u(\tau/2)-u^h(\tau/2)\Vert_{2\varepsilon_1}  \leq C \tau^{-1+\gamma/2} h^{2-2\varepsilon_1},
  \end{align*}
  According to the above estimate, there exists $h_1>0$ such that when $h<h_1$, 
  \begin{equation*}
      \|A_h^{\varepsilon_1}u^h({\tau/2})\|\leq M_2(\|u({\tau/2})\|_{2\varepsilon_1}+\|u({\tau/2})-u^h({\tau/2})\|_{2\varepsilon_1})\leq M_{u_0}. 
  \end{equation*}
  Reapplying \cite[Theorem 4.1]{ABsPEE}, we conclude that the problem \eqref{eq3} with $u^h({\tau/2})$ as the initial value has a unique solution on $[0, \tau]$. 
As problem \eqref{eq3} is autonomous, it has a solution with $P_h u_0$ as the initial value on $\left[0, {3\tau}/{2}\right]$. 
Each time, the local solution is extended over the fixed length $\tau/2$ of interval. So, by finite times, the extended interval can cover the given interval $[0,T]$. 
\qed
\end{proof}

\subsection{Some estimates for $u^h(t)$ and $f_h(u^h(t))$}

We now proceed to lay the groundwork for the temporal discretization error analysis in Section \ref{sec:fuldis_erran}. 

\begin{proposition}\label{pro:prop_for_U2}
  Let $u^h(t)$ be the solution of problem \eqref{eq3}. Then it is twice Fr$\acute{\text{e}}$chet differentiable in $(0,T]$, and adheres to the following estimates
  \begin{align}
      \Vert A_h^{s/2} D_t^m u^h(t)\Vert&\leq Ct^{-m+\gamma/2-s/2}+C, &&-2\leq s\leq 2\text{ , }m=0,1,2, &\label{Dt_est}\\
      \Vert A_h^{s/2}(u^h(t)-u^h(0))& \Vert \leq Ct^{\gamma/2-s/2}+C t , &&-2\leq s\leq \gamma.& \label{uh_diff}
  \end{align}
\end{proposition}

\begin{proof}
    We begin by considering \eqref{Dt_est}. From \eqref{bound_uh}, the formula \eqref{Dt_est} holds for $m=0$. 
    As in \cite[Lemma 3.7]{mukam}, we set $V(t)=t^l D_t^l u^h(t)$ with $l=1,2$, it holds that
    \begin{align*}
        D_t V(t) &= lt^{l-1} D_t^l u^h(t)+ t^l D_t^{l+1} u^h(t) \\
        &= lt^{l-1} D_t^l u^h(t) + t^l \left(A_h D_t^l u^h(t)+D_t^l f_h(u^h(t))\right)\\
        &= A_h V(t) + lt^{l-1} D_t^l u^h(t) + t^l D_t^l f_h(u^h(t)).
    \end{align*}  
    Therefore by variation-of-constants formula, we have
    \begin{equation}\label{Du_voc}
        D_t^l u^h(t)=t^{-l}\int_{0}^{t}S_h(t-\tau)(l\tau^{l-1}D_t^l u^h(\tau) +\tau^l D_\tau^l f_h(u^h(\tau)))\text{d}\tau.
    \end{equation} 
    According to the chain rule, we have
    \begin{equation}\label{chain_rule}
    \begin{aligned}
        D_t f_h(u^h(t)) &=D_u f_h(u^h(t)) D_t u^h(t),\\
        D_{tt} f_h(u^h(t)) &=D_{uu}f_h(u^h(t)) (D_t u^h(t), D_t u^h(t))+D_u f_h(u^h(t)) D_{tt} u^h(t).
	\end{aligned}  
    \end{equation}

    When $m=1$, $-2\leq s\leq 0$. Set $r=\max(s,-\frac{d}{2}+\varepsilon)$ and $r'=\frac{d}{2}+r-\hat{\gamma}+\varepsilon$ such that $(r,\hat{\gamma},r')\in \mathcal{B}$. It follows from \eqref{BoNLT} and \eqref{AhPh} that $\Vert A_h^{s/2}f_h(u^h(t))\Vert\leq Ct^{\gamma/2-r'/2}+C$. Then we have
    \begin{align*}
        \Vert A_h^{s/2} D_t u^h(t)\Vert&\leq  \Vert A_h^{1+s/2} u^h(t)\Vert +\Vert {A_h^{s/2} f_h}(u^h(t))\Vert\\
        &\leq C t^{-1+\gamma/2-s/2}+C.
    \end{align*}
    From the above bound, we know the case of $s=0$. Select sufficiently small $\varepsilon_1$ such that  $(-\frac{d}{2}+2\varepsilon_1,\hat{\gamma},0)\in\mathcal{B}$. For $0<s<2-\frac{d}{2}+2\varepsilon_1$, by \eqref{Du_voc} and \eqref{AhDufh}, we deduce that
    \begin{equation}\label{Du1_res}
        \begin{aligned}
            \quad\:\: \|A_h^{s/2} D_t u^h(t)\|&\leq t^{-1}\int_0^t  \Vert A_h^{s/2}S_h(t-\tau) D_\tau u^h(\tau)\Vert \text{d}\tau
            \\
            &\qquad +t^{-1}\int_0^t \tau \Vert A_h^{s/2} S_h(t-\tau) A_h^{d/4-\varepsilon_1} A_h^{-d/4+\varepsilon_1} D_u f_h(u^h(\tau))D_\tau u^h(\tau))\Vert \text{d}\tau
            \\
            &\leq  Ct^{-1+\gamma/2-s/2}+ C.
        \end{aligned}
    \end{equation}
    With the estimate above, we extend it to $s=2$ by inserting $A_h^{\varepsilon_1}$ in the first term on the right hand side and using $(\min(\varepsilon_1,2+\varepsilon_1-d),\hat{\gamma},2-{d/2}-\varepsilon_1)$, $(\varepsilon_1,\hat{(\gamma)},{d/2}-\varepsilon_1)\in\mathcal{B}$ for the second term gradually. 
    
    When $m=2$, $-2\leq s\leq 0$, we have $ \Vert A_h^{s/2}D_u f_h(u^h(t))D_t u^h(t)\Vert\leq Ct^{-1+\gamma/2-r'/2}$. Then it follow from \eqref{chain_rule} that
    \begin{equation}
        \begin{aligned}
            \Vert A_h^{s/2} D_t^2 u^h(t)\Vert&\leq  \Vert A_h^{1+s/2}D_t u^h(t)\Vert +\Vert A_h^{s/2} D_u f_h(u^h(t))D_t u^h(t)\Vert\\
            &\leq C t^{-2+\gamma/2-s/2}.
        \end{aligned}
    \end{equation}
    For $0<s\leq 2$, by \eqref{Du_voc} and \eqref{chain_rule}, we have
    \begin{equation}\label{Du2_res}
        \begin{aligned}
             \|A_h^{s/2} D_t^2 u^h(t)\|&\leq t^{-2}\int_0^t  \Vert 2\tau A_h^{s/2}S_h(t-\tau) D_\tau^2 u^h(\tau)\Vert \text{d}\tau
            \\
            &\qquad +t^{-2}\int_0^t \tau^2 \Vert A_h^{s/2-r/2} S_h(t-\tau)  A_h^{r/2}D_u f_h(u^h(\tau))D_\tau^2 u^h(\tau)\Vert \text{d}\tau.
            \\
            &\qquad +t^{-2}\int_0^t \tau^2 \Vert A_h^{s/2-\varepsilon_1} S_h(t-\tau)  A_h^{\varepsilon_1} D_{uu}f_h(u^h(\tau))(D_\tau u^h(\tau))^2\Vert \text{d}\tau.
        \end{aligned}
    \end{equation}
    The treatment of the first and second terms is the same as \eqref{Du1_res}. 
    Using \eqref{AhDuufh} with $(2\varepsilon_1,\hat{\gamma},d/2-\varepsilon_1)\in \mathcal{B}$ yields 
    $$
        \Vert A_h^{\varepsilon_1} D_{uu}f_h(u^h(\tau))(D_\tau u^h(\tau))^2 \Vert \leq C t^{-3+\gamma/2}.$$
    Substitute above bound into \eqref{Du2_res} completes the proof for the case of $m=2$. 

    Then we consider \eqref{uh_diff}, which can be immediately derived through Taylor expansion and \eqref{Dt_est}.
    \begin{align*}
        \Vert A_h^{s/2}(u^h(t)-u^h(0))& \Vert \leq t\Vert A_h^{s/2}\int_{0}^{1}D_t u^h(\xi t)\text{d}\xi\Vert \leq C t^{\gamma/2-s/2} +C t.
    \end{align*}
    This completes the proof.
    \qed
\end{proof}

Next we analyze the interpolation error of $f_h(u^h(t))$. 
\begin{proposition}\label{TaylorExp}
	Let $f_h(u^h(t))$ be the nonlinear term of problem \eqref{eq3}. Consider the interval $[t_1,t_2] \subset (0,T]$ with $\tau = t_2-t_1$ and choose $(s,\hat{\gamma},s_2)\in \mathcal{B}$. Then, for  $t \in [t_1,t_2]$, the following estimates hold
	\begin{align}
		\Vert A_h^{s/2} ( f_h(u^h(t))-f_h (u^h(t_1)))\Vert &\leq C t_1^{-1+{\gamma/2}-{s_2/2}}\tau,\label{f_diff1}
		\\
		\left\Vert A_h^{s/2} (f_h(u^h(t))-\frac{t_{2}-t}{\tau}f_h(u^h(t_1))\right.-\left.\frac{t-t_1}{\tau }f_h(u^h(t_2)))\right\Vert &\leq C t_1^{-2+\gamma/2-{s_2/2}}\tau^2. \label{f_diff2}
	\end{align}
\end{proposition}

\begin{proof}
	The proof is based on the Taylor expansion
	\begin{equation}
	\begin{aligned} 
		f_h(u^h(t)) = &\ f_h(u^h(t_1))+(t-t_1)D_t f_h(u^h(t_1)) \\
		&\ +(t-t_1)^2\int_{0}^{1}(1-\xi)D_{tt} f_h(u^h(t_1+\xi(t-t_1)))\text{d}\xi.\label{TayEx}
	\end{aligned}
    \end{equation}
	Using \eqref{AhDufh}, \eqref{Dt_est} and the chain rule \eqref{chain_rule}, we obtain
	\begin{equation}
        \Vert A_h^{s/2} D_t f_h(u^h(t))\Vert 
        \leq C(\Vert A_h^{\hat{\gamma}/2} u^h(t)\Vert+1)\Vert A_h^{s_2/2} D_t u^h(t)\Vert
        \leq C t^{-1+\gamma/2-s_2/2}.   \label{Dtg_e1}
	\end{equation}
	Furthermore, employing \eqref{AhDuufh} yields
	\begin{equation}
		\begin{aligned}
			\Vert A_h^{s/2} D_{tt} f_h(u^h(t))\Vert
			&\leq C\Vert A_h^{\hat{\gamma}/2} D_t u^h(t)\Vert \Vert A_h^{s_2/2} D_t u^h(t)\Vert + C(\Vert A_h^{\hat{\gamma}/2} u^h(t)\Vert+1)\Vert A_h^{s_2/2} D_{tt} u^h(t)\Vert_{s_2}
            \\ 
			&\leq Ct^{-2+\gamma/2-s_2/2+\gamma/2-\hat{\gamma}/2}+Ct^{-2+\gamma/2-s_2/2}+C\\
			&\leq Ct^{-2+\gamma/2-s_2/2}+C. 
		\end{aligned}\label{Dttg_e1}
	\end{equation}
	Substituting \eqref{Dtg_e1} and \eqref{Dttg_e1}  into \eqref{TayEx} leads to the desired results.
\end{proof}

\section{Error analysis for fully discrete scheme}\label{sec:fuldis_erran}
In this section, we develop a fully discrete numerical scheme for the problem \eqref{eq2} and provide error estimates for this scheme.
 
\subsection{Fully discrete scheme and error recursion}
Without loss of generality, we use a fixed time step $\Delta t=T/N$, $N\in \mathbb{N}^+$ and we set $t_k=k\Delta t$, $k=0,1,...,N$. The exponential Runge-Kutta method is adopted to numerically advance the semidiscrete form \eqref{eq3}  from $t_k$ to $t_{k+1}$. Note that the exact solution of \eqref{eq3} can be written as
\begin{align}
    u^h(t_{k+1})=S_h(\Delta t)u^h(t_k)+\int_{t_k}^{t_{k+1}}S_h(t_{k+1}-\tau)f_h(u^h(\tau))\:\text{d}\tau.\label{ERK_S_E}
\end{align}
Giving the numerical solution $u^h_k$ at time $t_k$, if we approximate $f_h(u^h(\tau))$ by $f_h(u^h_k)$ for $\tau\in[t_k, t_{k+1}]$ and calculate the resulting integral exactly, we obtain the first-order exponential Euler scheme
\begin{align}
    \hat{u}^h_{k+1}=S_h(\Delta t)u^h_k+\int_{t_k}^{t_{k+1}}S_h(t_{k+1}-\tau)f_h(u^h_k)\:\text{d}\tau.\label{ERK_S_inner}
\end{align}
The second-order exponential Runge-Kutta method can be obtained by approximating $f_h(u^h(\tau))$ through linear interpolation based on $f_h(u^h_k)$ and $f_h(\hat{u}^h_{k+1})$. Then, the second-order exponential Runge-Kutta method takes the following form
\begin{align}
    u^h_{k+1}=S_h(\Delta t)u^h_k+\int_{t_k}^{t_{k+1}}S_h(t_{k+1}-\tau)\left[\frac{t_{k+1}-\tau}{\Delta t}f_h(u^h_k)+\frac{\tau-t_k}{\Delta t}f_h(\hat{u}^h_{k+1})\right]\text{d}\tau.\label{ERK_S}
\end{align}
Numerical scheme \eqref{ERK_S_inner}-\eqref{ERK_S} constitutes our fully discrete method for problem \eqref{eq2}.  Conventionally, \eqref{ERK_S_inner} is referred to as the inner stage value of the scheme.

Denoting the temporal discretization error as $e_n=u^h_{n}-u^h(t_{n})$ and  subtracting \eqref{ERK_S} by \eqref{ERK_S_E} leads to
\begin{align*}
    e_n=S_h(\Delta t)e_{n-1}+\omega_{n}^{[1]}+\omega_{n}^{[2]}+\omega_{n}^{[3]},
\end{align*}
where $\omega_{n}^{[1]},\omega_{n}^{[2]}$ and $\omega_{n}^{[3]}$ are defined as follows
\begin{align*}
    &\omega_{n}^{[1]}=\int_{t_{n-1}}^{t_{n}}S_h(t_{n}-\tau)\frac{t_{n}-\tau}{\Delta t}\left(f_h(u^h_{n-1})-f_h(u^h(t_{n-1}))\right)\text{d}\tau,
    \\
    &\omega_{n}^{[2]}=\int_{t_{n-1}}^{t_{n}}S_h(t_{n}-\tau)\frac{\tau-t_{n-1}}{\Delta t}\left(f_h(\hat{u}^h_{n})-f_h(u^h(t_{n}))\right)\text{d}\tau,
    \\
    &\omega_{n}^{[3]}=\int_{t_{n-1}}^{t_{n}}S_h(t_{n}-\tau)\left[\frac{t_{n}-\tau}{\Delta t}f_h(u^h(t_{n-1}))+\frac{\tau-t_{n-1}}{\Delta t}f_h(u^h(t_{n})) - f_h(u^h(\tau))\right]\text{d}\tau.
\end{align*}
Solving this recursion with initial condition $e_0=0$ gives
\begin{equation}
    e_n=\sum_{k=0}^{n-1}S_h(t_{n-k-1})\omega_{k+1}^{[1]}+\sum_{k=0}^{n-1}S_h(t_{n-k-1})\omega_{k+1}^{[2]}+\sum_{k=0}^{n-1}S_h(t_{n-k-1})\omega_{k+1}^{[3]}.\label{error_recursion}
\end{equation}

\subsection{Error estimate} 
We proceed with the convergence analysis of the fully discrete scheme in the space $H^\mu$ ($0\leq \mu \leq \gamma$), investigating the order reduction induced by nonsmooth initial data. 
The stability of the numerical scheme relies on Lipschitz condition \eqref{localLips_discrete}, which necessitates boundedness of the discrete solution in an appropriate function space. 
We incorporate the boundedness analysis into convergence analysis. 
Through mathematical induction, we establish both boundedness and general $H^\mu$ estimates. Before stating the theorem, we define $\sigma(r)=\min(\hat{\gamma}+r-d/2,\:\hat{\gamma})-\varepsilon$ for $r\in [0,3/2)$, which ensures that $(-r,\hat{\gamma},-\sigma(r))\in \mathcal{B}$, while minimizing the value of $-\sigma(r)$.

\begin{theorem}\label{the:boundedness}
Let $u^h(t)$ and $u^h_n$ be the solution of problem \eqref{eq3} and discrete scheme \eqref
{ERK_S_inner}-\eqref{ERK_S} at time $t_n$.  
Let $ \mu \in [0,\max(1,\hat{\gamma})]$. Select $r\geq 0$ such that $\mu+r\leq 2$, $(-r,\hat{\gamma},\mu)\in \mathcal{B}$ and $\sigma(r)\geq 0$. 
Then there exists $\Delta t'>0$ such that for $\Delta t < \Delta t'$, we have
\begin{equation*}
    \|u^h_n-u^h(t_n)\|_\mu \leq C t_n^{-\mu/2-r/2}\Delta t^{\min(1+\gamma/2+\sigma(r)/2,\:2)}, \quad n=1,...,N,
\end{equation*}
where $C$ is independent of $h$, $\Delta t$ and $n$.  
\end{theorem} 

\begin{proof}
According to \eqref{bound_uh}, there exist $M>0$, such that $\Vert u^h(t)\Vert_{\hat{\gamma}}\leq M$. It is essential to prove that the discrete solution $u^h_n$ remains bounded in an appropriate function space. We use mathematical induction to prove this. 
First, for $k=0$, $\Vert u^h_0\Vert_{\hat{\gamma}}=\|u^h(0)\|_{\hat{\gamma}}\leq M+1$. 
Next, making the induction hypothesis that for $k\leq n-1$, $\Vert u^h_k\Vert_{\hat{\gamma}}\leq M+1$, 
we aim to show that this bound also holds for $k=n$ as long as $\Delta t$ is sufficiently small.
Noting that
\begin{equation*}
	\|u^h_n\|_{\hat{\gamma}}\leq\|u^h(t_n)\|_{\hat{\gamma}}+\|u^h_n-u^h(t_n)\|_{\hat{\gamma}},
\end{equation*}	
it suffices to demonstrate $\|u^h_n-u^h(t_n)\|_{\hat{\gamma}}$ converges.

Consider the first term.  Under induction hypothesis, it follows from \eqref{bound_u} and \eqref{localLips_discrete} that
\begin{align}\label{Lips_the2}
	\|A_h^{-r/2}(f_h(u^h_k)-f_h(u^h(t)))\|\leq C\|u^h_k-u^h(t)\|_{\mu},\qquad  k \leq n-1. 
\end{align}
Using above estimate and the semigroup properties, it holds that
\begin{equation}\label{ERK_E_1_R}
    \begin{aligned}
        \left\|\sum_{k=0}^{n-1}S_h(t_{n-k-1})\omega_{k+1}^{[1]}\right\|_\mu
        &\leq C\sum_{k=0}^{n-1}\int_{t_k}^{t_{k+1}}\Vert A_h^{\mu/2} S_h(t_n-\tau)A_h^{r/2}A_h^{-r/2}(f_h(u^h_k)-f_h(u^h(t_k)))\Vert \text{d}\tau
        \\
        &\leq C\sum_{k=0}^{n-1}\int_{t_k}^{t_{k+1}}(t_{n}-\tau)^{-\mu/2-r/2}\Vert e_k \Vert_\mu \text{d}\tau
        \\
        &\leq C\Delta t\sum_{k=0}^{n-2}t_{n-k-1}^{-{\mu}/2-r/2}\|e_k\|_\mu +C\Delta t^{1-\mu/2-r/2} \|e_{n-1}\|_\mu
        \\
        &\leq C\Delta t\sum_{k=1}^{n-1}t_{n-k}^{-\mu/2-r/2}\|e_k\|_\mu.
    \end{aligned}
\end{equation}	
	
Consider the second term of \eqref{error_recursion}. Note that under the induction hypothesis,
the inner stage value $\|\hat{u}^h_{k+1}\|_{\mu}$ is bounded for $k \leq n-1$.
Then we can bound $\omega^{[2]}_{k+1}$ by following the same procedure as for $\omega^{[1]}_{k+1}$.
	\begin{equation}
		\left\|\sum_{k=0}^{n-1}S_h(t_{n-k-1})\omega_{k+1}^{[2]}\right\|_{{\mu}}\leq C\Delta t\sum_{k=0}^{n-1}t_{n-k}^{-{\mu}/2-r/2}\|A_h^{-\sigma(r)/2}E_{k}\|\label{ERK_E_2}.
	\end{equation}
Here $E_k = \hat{u}^h_{k+1}-u^h(t_{k+1})$ is referred to as the inner stage error. The variation of constants formula yields
\begin{equation}
    \begin{aligned}E_{k}&=S_h(\Delta t)e_k+\int_{t_k}^{t_{k+1}}S_h(t_{k+1}-\tau)\left(f_h(u^h_{k})-f_h(u^h(\tau))\right) \text{d}\tau
        \\
        &=S_h(\Delta t)e_k+\int_{t_k}^{t_{k+1}}S_h(t_{k+1}-\tau)\left(f_h(u^h_{k})-f_h(u^h(t_k))\right) \text{d}\tau-\Delta_k,\end{aligned}\label{ERK_E_EK}
\end{equation}
where
\begin{equation*}
	\Delta_k=\int_{t_k}^{t_{k+1}}S_h(t_{k+1}-\tau)\left(f_h(u^h(\tau))-f_h(u^h(t_k))\right)\text{d}\tau.
\end{equation*}
The condition $\sigma(r)\geq 0$ implies that $\sigma(\sigma(r))$ is well-defined. 
Then we can bound $\Delta_k$ by using \eqref{uh_diff} and \eqref{f_diff1} with $s = -\sigma(r)$. Specifically, for $k=0$, 
\begin{equation*}
	\|A_h^{-\sigma(r)/2}\Delta_0\|\leq C\int_0^{\Delta t}  \tau^{\gamma/2+\sigma(\sigma(r))/2}\:\text{d}\tau\leq C\Delta t^{\min(1+\gamma/2+\sigma(\sigma(r))/2,\:2)}.
\end{equation*}
For $0<k<n$,
\begin{align*}
	\|A_h^{-\sigma(r)/2}\Delta_{k}\|&\leq C t_k^{-1+\gamma/2+\sigma(\sigma(r))/2}\Delta t^{2}.
\end{align*}
Returning to \eqref{ERK_E_EK} and using Lipschitz condition \eqref{Lips_the2} yields
\begin{equation}
\begin{aligned}
	\|A_h^{-\sigma(r)/2} E_0\| & \leq C\Delta t^{\min(1+\gamma/2+\sigma(\sigma(r))/2,\:2)}, \\
	\|A_h^{-\sigma(r)/2} E_k\| & \leq C(1+\Delta t^{1-r/2+\sigma(r)/2})\|e_k\|_{\mu}+ C t_k^{-1+\gamma/2+\sigma(\sigma(r))/2}\Delta t^{2},\; 0<k<n,
	\label{ERK_E_EK_R}
\end{aligned}
\end{equation}
where by definition $r-\sigma(r)<2$. 
Plugging \eqref{ERK_E_EK_R} into \eqref{ERK_E_2}, we obtain
\begin{equation}
\begin{aligned}
    \left\|\sum_{k=0}^{n-1}S_h(t_{n-k-1})\omega_{k+1}^{[2]}\right\|_{\mu}
    &\leq C\Delta t\sum_{k=1}^{n-1}t_{n-k}^{-{\mu}/2-r/2}\|e_{k}\|_{\mu}+C \Delta t^{3} \sum_{k=1}^{n-1}t_{n-k}^{-{\mu}/2-r/2}t_k^{-1+\gamma/2+\sigma(\sigma(r))/2}.
\end{aligned}\label{ERK_E_2_R}
\end{equation}

Finally, we consider the last term of \eqref{error_recursion}. For $k=0$, using \eqref{uh_diff} with $s=-r$, we have 
	\begin{equation}
		\begin{aligned}
			\|S_h(t_{n-1})\omega^{[3]}_1\|_{\mu}
			&\leq C\int_0^{\Delta t}\left(t_{n}-\tau\right)^{-\mu/2-r/2}\Delta t^{\min(\gamma/2+\sigma(r)/2,\:1)} \text{d}\tau\\
			&\leq C t_{n-1}^{-\mu/2-r/2} \Delta t^{\min(1+\gamma/2+\sigma(r)/2,\:2)}.
		\end{aligned}\label{ERK_E_3_1}
	\end{equation}
	For $0<k<n$, using \eqref{f_diff2} with $s=-r$, it holds that
	\begin{equation}
		\begin{aligned}
			\quad\:\:\left\|\sum_{k=1}^{n-1}S_h(t_{n-k-1})\omega^{[3]}_{k+1}\right\|_{\mu}
			&\leq C\sum_{k=1}^{n-1}\int_{t_{k}}^{t_{k+1}}\left(t_{n}-\tau\right)^{-{\mu}/2-r/2}\Delta t^{2}t_{k}^{-2+\gamma/2+\sigma(r)/2}\text{d}\tau
			\\
			&\leq C\Delta t^{3}\sum_{k=1}^{n-2}t_{n-k-1}^{-\mu/2-r/2} t_k^{-2+\gamma/2+\sigma(r)/2}+ C \Delta t^{3-\mu/2-r/2} t_{n-1}^{-2+\gamma/2+\sigma(r)/2}
            \\
            &\leq C\Delta t^{3}\sum_{k=1}^{n-1}t_{n-k}^{-\mu/2-r/2} t_k^{-2+\gamma/2+\sigma(r)/2}.
		\end{aligned}\label{ERK_E_3_2}
	\end{equation}
	The bound \eqref{ERK_E_3_1} can be merged with the above estimate.

	Combining the estimates \eqref{ERK_E_1_R}, \eqref{ERK_E_2_R} and \eqref{ERK_E_3_2}, using \cite[Lemma 6.1]{IRK2} to bound the summation terms, it follows that
	\begin{align*}
		\|e_n\|_{\mu}&\leq C\Delta t\sum_{k=1}^{n-1}t_{n-k}^{-\mu/2-r/2}\|e_k\|_{\mu}+ Ct_n^{-\mu/2-r/2} \Delta t^{\min(1+\gamma/2+\sigma(r)/2,\:2)},
	\end{align*}
    where the second term on the right hand side of \eqref{ERK_E_2_R} is absorbed due to $\sigma(r)-\sigma(\sigma(r))<2$. 
	Applying the Gronwall inequality (see \cite[Lemma 6.2]{IRK2}) yields
	\begin{align*}
		\|e_n\|_{\mu} = \|u^h_n-u^h(t_n)\|_{\mu} \leq Ct_n^{-\mu/2-r/2} \Delta t^{\min(1+\gamma/2+\sigma(r)/2,\:2)}.
	\end{align*}
    When $\mu=\hat{\gamma}$, we set $r=\min(2-\hat{\gamma},3/2-\varepsilon)$. It follows from the definition of $\sigma$ that $1+\frac{\gamma}{2}+\frac{\sigma(r)}{2}-\frac{\hat{\gamma}}{2}-\frac{r}{2}>0$. This guarantees convergence of $\|u^h_n-u^h(t_n)\|_{\hat{\gamma}}$ and therefore the numerical solution remains bounded in $H^{\hat{\gamma}}$. 
    The proof is completed. 
    \qed
\end{proof}

\subsection{Optimal error estimates in $L^2$ and $H^1$ norms}
According to Theorem \ref{the:boundedness}, to achieve sharp convergence rates in both $L^2$ and $H^1$ norms, the parameter $r$ must be chosen to maximize $\sigma(r)$ while simultaneously minimizing $r$ itself. To satisfy this requirement we introduce
\begin{equation}\label{def_q}
    q=\left\{\begin{aligned}
        &\min(-d/2,-\gamma)+\varepsilon,  &&0< \gamma\leq 1,&\\
        &\min(-2+\gamma,-2+\gamma+\hat{\gamma}-d/2-\varepsilon), && 1< \gamma< 2,&
    \end{aligned}\right.
\end{equation}
which provides an approximately optimal to $-r$. 
Moreover, the condition $r + \mu < 2$ in Theorem \ref{the:boundedness} becomes restrictive when $\mu = 1$ and $d=3$. 
We overcome this limitation in the next theorem. 
\begin{theorem}\label{the:tempo_err}
Let $u^h(t)$ and $u^h_n$ denote the solutions to problem \eqref{eq3} and the discrete scheme \eqref{ERK_S_inner}–\eqref{ERK_S} at time  $t_n$, respectively. Then there exists $\Delta t'>0$ such that for $\Delta t < \Delta t'$, we have
\begin{align}
    \|u^h_n-u^h(t_n)\| &\leq C t_n^{q/2}\Delta t^{\min(1+\gamma-\varepsilon,\:2)}, \qquad t_n\in (0,T],\label{tempo_optErr_L2}\\
    \|u^h_n-u^h(t_n)\|_1 &\leq C t_n^{q/2-1/2}\Delta t^{\min(1+\gamma-\varepsilon,\:2)}, \qquad t_n\in (0,T], \label{tempo_optErr_H1}
\end{align}
where $C$ is independent of $h$, $\Delta t$ and $n$, and $\varepsilon$ denotes a sufficiently small positive number.
\end{theorem}

\begin{proof}
    Using Theorem \ref{the:boundedness} with $\mu=0$ and $r=-q$ yields formula \eqref{tempo_optErr_L2}. When $q-1>-2$, the formula \eqref{tempo_optErr_H1} follows from Theorem \ref{the:boundedness} with $\mu=1$ and $r=-q$. 
    It remains to prove the case of $q-1\leq -2$. 
    
    In this case, according to the definitions of $q$ and $\sigma(r)$, it follows that
    \begin{equation}\label{the3_case2_condi}
        d=3 \text{ , } 0<\gamma\leq\frac{5}{4}\text{ , }  \sigma(r)=\gamma+r-\frac{3}{2}-\varepsilon. 
    \end{equation}
    Denote $\omega_n=\omega_{n}^{[1]}+\omega_{n}^{[2]}+\omega_{n}^{[3]}$. Following an error recursion analogous to \eqref{error_recursion}, we obtain
    \begin{align*}
        t_n e_n&=S_h(\Delta t)t_{n-1}e_{n-1}+\Delta t S_h(\Delta t)e_{n-1}+t_n\omega_{n}\\
        &=\Delta t \sum^{n-1}_{k=1} S_h(t_{n-k})e_k + \sum_{k=0}^{n-1}  S_h(t_{n-k-1}) t_{k+1} \omega_{k+1}.
    \end{align*}
    Using the bound \eqref{tempo_optErr_L2} yields
    \begin{equation*}
        \begin{aligned}
            \left\Vert \Delta t \sum^{n-1}_{k=1} S_h(t_{n-k})e_k \right\Vert_1 &\leq C \Delta t \sum^{n-1}_{k=1} t_{n-k}^{-1/2} t_k^{q/2}\Delta t^{\min(1+\gamma-\varepsilon,2)}
            \\
            &\leq C t_n^{q/2+1/2}\Delta t^{\min(1+\gamma-\varepsilon,2)}.
        \end{aligned}
    \end{equation*}
    Select $r_1>0$ such that $(-r_1,\hat{\gamma},1)\in \mathcal{B}$, $r_1+1<2$ and $-\sigma(r_1)\geq 0$. 
    We need to bound $\Delta_k$ in \eqref{ERK_E_EK}. Here $\sigma(\sigma(r_1))$ may not be well-defined, so we absorb $A_h^{-\sigma(r_1)/2}$ by the semigroup properties. 
    Using \eqref{uh_diff} and \eqref{f_diff1} with $s = 0$, we get
    \begin{align*}
        \|A_h^{-\sigma(r_1)/2}\Delta_0\|\leq C\int_0^{\Delta t}& (t_1-\tau)^{\sigma(r_1)/2} \tau^{\gamma/2+\sigma(0)/2}\:\text{d}\tau\leq C\Delta t^{\min(1+\gamma/2+\sigma(r_1)/2+\sigma(0)/2,\:2)},\\
        \|A_h^{-\sigma(r_1)/2}\Delta_{k}\|&\leq C t_k^{-1+\gamma/2+\sigma(0)/2}\Delta t^{2+\sigma(r_1)/2}, \qquad 0<k<n.
    \end{align*}
    Then following the same procedure for deriving \eqref{ERK_E_2_R}, we arrive at
    \begin{equation*}
    \begin{aligned}
        \left\|\sum_{k=0}^{n-1}S_h(t_{n-k-1})t_{k+1}\omega_{k+1}^{[2]}\right\|_{1}
        &\leq C\Delta t\sum_{k=1}^{n-1}t_{n-k}^{-{1}/2-r_1/2}t_k\|e_{k}\|_{1}+C \Delta t^{3+\sigma(r_1)/2} \sum_{k=1}^{n-1}t_{n-k}^{-{1}/2-r_1/2}t_k^{\gamma/2+\sigma(0)/2}.
    \end{aligned}
    \end{equation*}
    Combining the estimates \eqref{ERK_E_1_R} and \eqref{ERK_E_3_2} with $\mu=1$ and $r=r_1$ yields 
    \begin{align*}
        \Vert t_n e_n\Vert_1
        &\leq C t_n^{q/2+1/2}\Delta t^{\min(1+\gamma-\varepsilon,2)} + C\Delta t\sum_{k=1}^{n-1}t_{n-k}^{-1/2-r_1/2}t_k\|e_k\|_1
        \\ 
        &\quad + C \Delta t^{3+\sigma(r_1)/2} \sum_{k=1}^{n-1}t_{n-k}^{-{1}/2-r_1/2}t_k^{\gamma/2+\sigma(0)/2} + C\Delta t^{3}\sum_{k=1}^{n-1}t_{n-k}^{-1/2-r_1/2} t_k^{-1+\gamma/2+\sigma(r_1)/2}. 
    \end{align*}
    Substituting condition \eqref{the3_case2_condi} into the above estimate and employing \cite[Lemma 6.1]{IRK2} to bound the summation terms, we observe that the last two terms can be absorbed by the first term. Then applying the Gronwall inequality completes the proof.
\end{proof}

\begin{remark}
Combining Theorems \ref{thm:space_err} and \ref{the:tempo_err}, we derive the fully discrete error bounds \eqref{full_err} in both $L^2$ and $H^1$ norms. For context, prior studies \cite{10He,18_Burges} have shown that the second-order Crank–Nicolson/Adams–Bashforth scheme, when applied to two-dimensional Navier–Stokes and Burgers equations with $H^1$ initial data, attains temporal convergence orders of 1.5 and 1 in $L^2$ and $H^1$ norms, respectively. While direct comparisons must be interpreted with caution due to differing discretization formulations, our scheme demonstrates nominal convergence rate improvements of 1/2 order in the $L^2$ norm and 1 order in the $H^1$ norm. This discrepancy may stem from the comparatively weaker nonlinearity inherent in the Field–Noyes model relative to these convection-dominated systems. 
\end{remark}

\section{Numerical experiments}\label{sec:experiments}
In this section, we present numerical tests to support the theoretical analysis. 
We consider the Field-Noyes model \eqref{eq:eq1} with $a_1=a_2=a_3=1$, $\lambda = \delta = 0.1$, $\rho = 0.25$, and $c = 1$ 
in the unit square $\Omega = (0,1)\times (0,1)$ up to $T = 0.1$. 
The linear Galerkin finite element method \eqref{eq3} is used for spatial discretization, and the second-order exponential Runge-Kutta method \eqref{ERK_S_inner}-\eqref{ERK_S} for temporal discretization. 
By Theorems \ref{thm:space_err} and \ref{the:tempo_err}, we obtain the errors of the fully discrete scheme in both $L^2$ and $H^1$ norms: 
\begin{align}\label{full_err_d2}
    \Vert u(t_n)-u^h_n \Vert&\leq C t_n^{-1+\gamma/2}h^2 + C t_n^{\max(-1+\varepsilon,\: \gamma-2)/2} \Delta t^{\min(1+\gamma-\varepsilon,\:2)}, \quad 0<t_n\leq T,
    \\
    \Vert  u(t_n)-u^h_n \Vert_1&\leq C t_n^{-1+\gamma/2}h + C t_n^{\max(-1+\varepsilon,\: \gamma-2)/2-1/2} \Delta t^{\min(1+\gamma-\varepsilon,\:2)},\quad 0<t_n\leq T.
\end{align}
To illustrate the influence of initial value regularity on the convergence order, we separately consider the following initial values:
\begin{enumerate}[label=(\roman*).]
    \item $u_{0,1}(x_1,x_2)=0.5\:\text{sgn}(x_2-0.5)+0.5\in D(A^{1/4-\varepsilon})$.
    \item $u_{0,1}(x_1,x_2)=(x_1^2+x_2^2)^{-1/8}-0.8\in D(A^{3/8-\varepsilon})$.
    \item $u_{0,1}(x_1,x_2)=x^{1/2}_1 x_2\in D(A^{1/2-\varepsilon})$.
    \item $u_{0,1}(x_1,x_2)=|x_2-x_1|\in  D(A^{3/4-\varepsilon})$.
\end{enumerate}
Here, $\text{sgn}(x)$ denotes the sign function. In all numerical tests, we take    $u_{0,2}=u_{0,3}=u_{0,1}$.

\begin{table}[b]
	\centering
	\caption{Spatial discretization errors $\Vert u_{ref}(T)-u^h_N\Vert_\alpha\:(\alpha=0,1)$ and convergence orders for the Field-Noyes model with initial values (\rmnum{1})-(\rmnum{4}). Top: $L^2$ norm errors; bottom: $H^1$ norm errors. Theoretical convergence orders are listed in the last row.}
    \label{tab:spa_con_rate}
	\begin{tabular}{ccccccccccccc}
		\hline
		\multirow{2}{*}{$h$} & & \multicolumn{2}{c}{Initial data (\rmnum{1})} &&  \multicolumn{2}{c}{Initial data (\rmnum{2})}  && \multicolumn{2}{c}{Initial data (\rmnum{3}) } && \multicolumn{2}{c}{Initial data (\rmnum{4})}
		\\ \cline{3-4} \cline{6-7} \cline{9-10} \cline{12-13} 
		~& & Error & Order &&  Error & Order &&  Error & Order &&  Error & Order 
		\\ \hline
		$1/2^2$ & ~ & 1.652E-02 & -- & ~ & 5.599E-03 & -- & ~ & 8.649E-03 & -- & ~ & 1.187E-02 & -- \\ 
        $1/2^3$ & ~ & 3.995E-03 & 2.048  & ~ & 1.405E-03 & 1.994   & ~ & 2.157E-03 & 2.003  & ~ & 3.121E-03 & 1.926   \\
        $1/2^4$ & ~ & 9.506E-04 & 2.071  & ~ & 3.385E-04 & 2.054   & ~ & 5.182E-04 & 2.058  & ~ & 7.528E-04 & 2.052   \\
        -- & &-- & 2 & & --& 2 & &-- & 2 & & --& 2 \\
        \hline
        $1/2^2$ & ~ & 2.011E-01 & -- & ~ & 6.908E-02 & -- & ~ & 1.093E-01 & -- & ~ & 9.452E-02 & -- \\ 
        $1/2^3$ & ~ & 9.209E-02 & 1.127  & ~ & 3.169E-02 & 1.124   & ~ & 5.006E-02 & 1.127  & ~ & 4.280E-02 & 1.143   \\
        $1/2^4$ & ~ & 4.316E-02 & 1.093  & ~ & 1.485E-02 & 1.093   & ~ & 2.345E-02 & 1.094  & ~ & 1.968E-02 & 1.121   \\ 
        -- & & --& 1 & & --& 1 & & --& 1 & & --& 1 \\
        \hline
	\end{tabular}
\end{table}

We first compute the convergence orders in both space and time. The reference solution $u_{ref}(t)$ is obtained using a adequately fine spatial mesh with $h=1/64$ and an sufficiently small stepsize $\Delta t=1/40960$, ensuring reliable approximations to both the exact solution $u(t)$ and its semidiscrete counterpart $u^h(t)$. 
For spatial aspect, Table \ref{tab:spa_con_rate} presents the errors $\Vert u_{ref}(T)-u^h_N\Vert_\alpha\:(\alpha=0,1)$ and convergence orders, where the numerical solution $u^h_n$ is computed using identical temporal step sizes to the reference solution. 
For temporal aspect, Table \ref{tab:temp_con_rate} shows the errors $\Vert u_{ref}(T)-u^h_N\Vert_\alpha$ and corresponding convergence orders, where the spatial mesh of numerical solution is identical to the one used for the reference solution. 
Both experiments demonstrate the sharpness of the error estimates.

Furthermore, to illustrate the convergence behaviours near $t = 0$ where the weak singularity may arise, we show the errors $\Vert u_{ref}(t_1)-u^h_1\Vert$ and temporal
convergence orders in Table \ref{tab:uniform}. According to estimate \eqref{full_err_d2}, the first-step errors achieve convergence orders of $\max(-1+\varepsilon, \gamma-2)/2+{\min(1+\gamma-\varepsilon,\:2)}$ and $\max(-1+\varepsilon, \gamma-2)/2-1/2+{\min(1+\gamma-\varepsilon,\:2)}$ in the $L^2$ and $H^1$ norms, respectively. The numerical results show close agreement with theory for initial values (\rmnum{2})-(\rmnum{4}), but slightly higher rates for initial value (\rmnum{1}). 
This discrepancy can be explained by examining the nonlinearity estimates. 

Indeed, for estimating the Fréchet derivatives of the nonlinearity in \eqref{AhDufh} and \eqref{AhDuufh}, we employ pointwise multiplication relationships in fractional Sobolev spaces. 
When $s\in(-d/2, d/2)$, the power function $|x|^s$ satisfies these estimates, explaining the good agreement between theory and numerical results for initial data (\rmnum{2}) and (\rmnum{3}). However, for initial data (\rmnum{1}), where the generic case would expect $(0,1/2,1/2)\in \mathcal{B}$ but here we obtain $(1/2,1/2,1/2)$ due to the product of two sign functions remaining a sign function, this leads to improved estimates that account for the observed higher convergence rates.

\begin{table}[t]
	\centering
	\caption{Temporal discretization errors  $\Vert u_{ref}(T)-u^h_N\Vert_\alpha\:(\alpha=0,1)$ and convergence orders for the Field-Noyes  model with initial values (\rmnum{1})-(\rmnum{4}).  Top: $L^2$ norm errors; bottom: $H^1$ norm errors. Theoretical convergence orders are listed in the last row.}
    \label{tab:temp_con_rate}
    \begin{adjustbox}{max width=\textwidth}
	\begin{tabular}{ccccccccccccc}
		\hline
		\multirow{2}{*}{$N$} & & \multicolumn{2}{c}{Initial data (\rmnum{1})} &&  \multicolumn{2}{c}{Initial data (\rmnum{2})}  && \multicolumn{2}{c}{Initial data (\rmnum{3}) } && \multicolumn{2}{c}{Initial data (\rmnum{4})}
		\\ \cline{3-4} \cline{6-7} \cline{9-10} \cline{12-13} 
		~& & Error & Order &&  Error & Order &&  Error & Order &&  Error & Order 
		\\ \hline
		$2^5$ && 3.588E-04 & -- & ~ & 3.015E-05 & -- & ~ & 3.423E-05 & -- & ~ & 5.886E-05 & --  \\ 
        $2^6$ && 1.225E-04 & 1.551  & ~ & 8.525E-06 & 1.822 & ~ & 8.858E-06 & 1.950 & ~ & 1.513E-05 & 1.960 \\ 
        $2^7$ && 4.098E-05 & 1.579  & ~ & 2.436E-06 & 1.807 & ~ & 2.280E-06 & 1.958 & ~ & 3.855E-06 & 1.973 \\ 
        $2^8$ && 1.331E-05 & 1.622  & ~ & 6.988E-07 & 1.801 & ~ & 5.842E-07 & 1.965 & ~ & 9.742E-07 & 1.984 \\
        -- & &-- & $1.5-\varepsilon$ & &-- & $1.75-\varepsilon$ & &-- & $2-\varepsilon$ & &-- & $2-\varepsilon$ \\        
        \hline
        $2^5$ && 7.320E-04 & -- & ~ & 6.154E-05 & -- & ~ & 7.060E-05 & -- & ~ & 8.517E-05 & -- \\ 
        $2^6$ && 2.465E-04 & 1.570   & ~ & 1.827E-05 & 1.752   & ~ & 1.818E-05 & 1.957   & ~ & 2.204E-05 & 1.950   \\ 
        $2^7$ && 8.167E-05 & 1.593   & ~ & 5.463E-06 & 1.741   & ~ & 4.659E-06 & 1.965   & ~ & 5.638E-06 & 1.967   \\ 
        $2^8$ && 2.635E-05 & 1.632   & ~ & 1.632E-06 & 1.743   & ~ & 1.188E-06 & 1.971   & ~ & 1.428E-06 & 1.981   \\
        -- & & -- & $1.5-\varepsilon$ & & --& $1.75-\varepsilon$ & &-- & $2-\varepsilon$ & &-- & $2-\varepsilon$ \\ 
        \hline
	\end{tabular}
    \end{adjustbox}
\end{table}

\begin{table}[tbhp]
	\centering
	\caption{The first step errors $\Vert u_{ref}(t_1)-u^h_1\Vert_\alpha\:(\alpha=0,1)$ and convergence orders for the Field-Noyes  model with initial values (\rmnum{1})-(\rmnum{4}). Top: $L^2$ norm errors; bottom: $H^1$ norm errors. Theoretical convergence orders are listed in the last row.}
    \label{tab:uniform}
    \begin{adjustbox}{max width=\textwidth}
	\begin{tabular}{ccccccccccccc}
		\hline
		\multirow{2}{*}{$\Delta t$} & & \multicolumn{2}{c}{Initial data (\rmnum{1})} &&  \multicolumn{2}{c}{Initial data (\rmnum{2})}  && \multicolumn{2}{c}{Initial data (\rmnum{3}) } && \multicolumn{2}{c}{Initial data (\rmnum{4})}
		\\ \cline{3-4} \cline{6-7} \cline{9-10} \cline{12-13} 
		~& & Error & Order &&  Error & Order &&  Error & Order &&  Error & Order 
		\\ \hline
		$0.1/2^7$ & ~ & 3.334E-03 & \makebox[0.05\textwidth][c]{--} & ~ & 8.009E-04 &  \makebox[0.05\textwidth][c]{--}  & ~ & 1.591E-04 & \makebox[0.03\textwidth][c]{--} & ~ & 2.997E-04 &  \makebox[0.05\textwidth][c]{--}  \\ 
        $0.1/2^8$ & ~ & 1.409E-03 & 1.242  & ~ & 3.754E-04 & 1.093  & ~ & 5.010E-05 & 1.667  & ~ & 8.236E-05 & 1.864   \\ 
        $0.1/2^9$ & ~ & 5.788E-04 & 1.284  & ~ & 1.712E-04 & 1.133  & ~ & 1.625E-05 & 1.624  & ~ & 2.340E-05 & 1.815   \\ 
        $0.1/2^{10}$ & ~ & 2.312E-04 & 1.324  & ~ & 7.630E-05 & 1.166  & ~ & 5.413E-06 & 1.586  & ~ & 6.838E-06 & 1.775  \\
        -- & &-- & $1-\varepsilon$ & & --& $1.25-\varepsilon$ & & --& $1.5-\varepsilon$ & & -- & $1.75-\varepsilon$ \\ 
         \hline
        $0.1/2^7$ & ~ & 2.953E-02 & \makebox[0.05\textwidth][c]{--}   & ~ & 1.099E-02 & \makebox[0.05\textwidth][c]{--}   & ~ & 1.859E-03 &\makebox[0.05\textwidth][c]{--}   & ~ & 3.295E-03 & \makebox[0.05\textwidth][c]{--}    \\ 
        $0.1/2^8$ & ~ & 1.694E-02 & 0.802  & ~ & 7.236E-03 & 0.602  & ~ & 8.641E-04 & 1.105  & ~ & 1.402E-03 & 1.233   \\
        $0.1/2^9$ & ~ & 9.608E-03 & 0.818  & ~ & 4.652E-03 & 0.637  & ~ & 4.102E-04 & 1.075  & ~ & 5.961E-04 & 1.233   \\ 
        $0.1/2^{10}$ & ~ & 5.334E-03 & 0.849  & ~ & 2.932E-03 & 0.666  & ~ & 2.010E-04 & 1.029  & ~ & 2.561E-04 & 1.219   \\
        -- & & --& $0.5-\varepsilon$ & &-- & $0.75-\varepsilon$ & &-- & $1-\varepsilon$ & & --& $1.25-\varepsilon$ \\ 
        \hline
	\end{tabular}
    \end{adjustbox}
\end{table}

\section{Conclusions}
\label{sec:conclusions}
In this study, we have investigated the fully discrete error of the exponential Runge-Kutta Galerkin finite element method for the Field-Noyes model with nonsmooth initial data. 
Our analysis is developed within the fractional Sobolev spaces $H^s$ for $0\leq s\leq2$. 
We first extend the regularity of the discrete space from  $H^1$ to $H^{3/2-\varepsilon}$ and establish the norm equivalence $\Vert A^{s/2}\cdot \Vert\sim\Vert  A_h^{s/2}\cdot \Vert$ in $X_h$ for $-\frac{3}{2}<s<\frac{3}{2}$. 
Using pointwise multiplication techniques in fractional Sobolev spaces $H^s$, we derive corresponding results in fractional operator domains $D(A_h^{s/2})$, including dual cases. These developments enable us to prove the key analysis \eqref{localLips_discrete}-\eqref{AhDuufh}, which serve as the theoretical foundation for obtaining sharp estimates. Finally, employing semigroup theory, we derive fully discrete error estimates in Theorem \ref{thm:space_err} and \ref{the:tempo_err}. The numerical example also supports the theoretical result. 

Remarkably, the error analysis in this work rely solely on the nonlinearity estimates \eqref{localLips_discrete}-\eqref{AhDuufh} and the properties of the linear operator $A$, without requiring additional specifics of the Field-Noyes model. 
This demonstrates that our framework can be directly applied to a class of abstract semilinear parabolic equations. 
In future, we intend to generalize our analysis to encompass broader classes of nonlinear operators,  and establish precise assumptions for Fréchet differentiability of the nonlinearity, which will enable the derivation of sharp error estimates. 

Additionally, Theorem \ref{the:tempo_err} reveals that when the initial regularity $\gamma > 1$, the temporal convergence order reaches the method's theoretical upper bound. A fundamental open question concerns whether the $(1+\gamma)$-order convergence can be preserved when employing higher-order exponential Runge-Kutta methods --- this constitutes another focus of our future research.

\section*{Data availability}
No data was used for the research described in the article.

\section*{Acknowledgments}
This work was supported in part by the National Natural Science Foundation of China (No.12001115).

\appendix

\section{Technical Lemmas}\label{sec:appendix}
\begin{lemmaA}\label{lem:3dcase}
  Let $\Omega\subset \mathbb{R}^3$ be a convex polygonal domain. Discrete it by a finite triangulation $\pi_h$ with maximal length $h$. Assume that the triangulation $\pi_{h}$ is regular and satisfies an inverse condition in the sense that
  \begin{equation}
      c^{-1}h\leq2\rho_\sigma\leq 2r_\sigma\leq ch,\quad\forall\sigma\in\pi_h,
      \label{RTReg}
  \end{equation}
  where $\rho_\sigma$ and $r_\sigma$ are the radii of inscribed and circumscribed spheres of $\sigma$, respectively. 
  Let $V_h\subset H^1$ denote the space of continuous and piecewise linear functions over the triangulation $\pi_h$. 
  Then for $0\leq s<\frac32$, $V_h$ is continuously embedded in $H^s$ with the estimate
  \begin{equation}
      \|u_h\|_s\leq C_s h^{-s}\|u_h\|,\qquad \forall u_h\in V_h. \label{A.1_res}
  \end{equation}
\end{lemmaA}

\begin{proof}
  When $s=0$, \eqref{A.1_res} is trivial. When $s=1$, \eqref{A.1_res} is very standard (see, e.g., \cite[Theorem 3.2.6]{FEMoEP}). Consequently, the result follows immediately for $0\leq s\leq1$ by the interpolation technology and the estimate $\|u_h\|_{s}\leq C\|u_h\|_{1}^{s}\|u_h\|^{1-s}.$

It therefore suffices to consider the case where $s=1+\epsilon$, $0<\epsilon<\frac12$. We can estimate directly the $H^{1+\epsilon}$ norm of ${u}_h\in V_h$ using the fact that the $H^{1+\epsilon}$ norm is given by a sum of integrals over $\Omega\times\Omega$. Indeed, we have 
\begin{align*}
    \|u_h\|_{1+\epsilon}^{2}=\|u_h\|^{2}+\sum_{j=1}^{3}\iint_{\Omega\times\Omega}\frac{|\partial_{j}u_h(x)-\partial_{j}u_h(y)|^{2}}{|x-y|^{3+2\epsilon}}\text{d}x\text{d}y.
\end{align*}
Since ${u}_{\sigma j}=(\partial_j{u}_h)|_\sigma$ is constant in each $\sigma\in\pi_h$, it follows that
\begin{align*}
    \iint_{\Omega\times\Omega}\frac{|\partial_{j}u_h(x)-\partial_{j}u_h(y)|^{2}}{|x-y|^{3+2\epsilon}}\text{d}x\text{d}y 
    =\:&\sum_{\sigma,\sigma^{\prime}\in\pi_h,\sigma^{\prime}\neq\sigma}|{u}_{\sigma j}-{u}_{\sigma^{\prime}j}|^{2}\iint_{\sigma\times\sigma^{\prime}}\frac{\text{d}x\text{d}y}{|x-y|^{3+2\epsilon}}
    \\
    \leq\:&\sum_{\sigma,\sigma^{\prime}\in\pi_h}2\big(|{u}_{\sigma j}|^{2}+|{u}_{\sigma^{\prime}j}|^{2}\big)\iint_{\sigma\times\sigma^{\prime}}\frac{\text{d}x\text{d}y}{|x-y|^{3+2\epsilon}} 
    \\
    \leq\:& 4\sum_{\sigma\in\pi_h}|{u}_{\sigma j}|^{2}\iint_{\sigma\times(\Omega\setminus\sigma)}\frac{\text{d}x\text{d}y}{|x-y|^{3+2\epsilon}}.
\end{align*}
It is then sufficient to verify that
\begin{equation}\label{A.1_midres}
  \iint_{\sigma\times(\Omega\setminus\sigma)}\frac{\text{d}x\text{d}y}{|x-y|^{3+2\epsilon}}\leq C_\epsilon r_\sigma^{3-2\epsilon}.
\end{equation}
Since this, together with \eqref{RTReg}, implies that
  \begin{align*}
      \| u_{h}\|_{H^{1+\epsilon}}^{2}& \leq\|u_h\|^{2}+C_{\epsilon}\sum_{j}\sum_{\sigma}|{u}_{\sigma j}|^{2}r_{\sigma}^{3-2\epsilon}  \\
      &\leq\|u_h\|_{L_{2}}^{2}+C_{\epsilon}h^{-2\epsilon}\sum_{j}\sum_{\sigma}|{u}_{\sigma j}|^{2}|\sigma|\\
      &\leq C_{\epsilon}h^{-2\epsilon}\|u_h\|_1^{2}\leq C_\epsilon(h^{-1-\epsilon}\|u_h\|)^2,
  \end{align*}
where $|\sigma|$ denotes the volume of $\sigma$, and hence implies the desired estimate.

Let us now show estimate \eqref{A.1_midres}. Let $B_{\sigma}$ denote the circumscribed sphere of $\sigma$ with radius $r_\sigma$, and $U_{\sigma}$ the spherical domain with the same center as $B_{\sigma}$ but with radius $2r_{\sigma\cdot}$ Let $S_i$ for $i=1,2,3,4$ be the four planes in $\mathbb{R}^3$ obtained by extending the faces of $\sigma$. Each $S_i$ partitions $U_{\sigma}$ into two subsets, so denote by $V_i$ the one disjoint with $\sigma$. Then, we obviously see that
$$
    \iint_{\sigma\times(\Omega\setminus\sigma)}\frac{\text{d}x\text{d}y}{|x-y|^{3+2\epsilon}}\leq\iint_{B_\sigma\times(\mathbb{R}^3\setminus U_\sigma)}\frac{\text{d}x\text{d}y}{|x-y|^{3+2\epsilon}}+\sum_{i=1}^4\iint_{\sigma\times V_i}\frac{\text{d}x\text{d}y}{|x-y|^{3+2\epsilon}}.$$

  For the first integral, we have
  \begin{align*}
      &\quad\:\:\iint_{B_{\sigma}\times(\mathbb{R}^{3}\setminus U_{\sigma})}\frac{\text{d}x\text{d}y}{|x-y|^{3+2\epsilon}}\\
      & =\int_{0}^{r_{\sigma}}\int_{2r_{\sigma}}^{\infty}\int_{0}^{2\pi}\int_{0}^{2\pi}\int_{0}^{2\pi}\int_{0}^{2\pi}\frac{r_1^2 \sin{\varphi_1}r_2^2\sin{\varphi_2} \text{d}\varphi_1 \text{d}\varphi_2 \text{d}\theta_1 \text{d}\theta_2 \text{d}r_2 \text{d}r_1}{[r_1^{2}+r_2^{2}-2r_1 r_2 H(\varphi_1,\varphi_2,\theta_1,\theta_2)]^{\frac32+\epsilon}}
      \\
      &\leq Cr_{\sigma}^{3-2\epsilon}\int_{2}^{\infty}\int_{0}^{1}\frac{r_1^2 r_2^2 \text{d}r_1 \text{d}r_2}{(r_2-r_1)^{3+2\epsilon}}\leq C_{\epsilon}r_{\sigma}^{3-2\epsilon},
  \end{align*}
  where $H(\varphi_1,\varphi_2,\theta_1,\theta_2)=\sin{\varphi_1}\sin{\varphi_2}\cos{(\theta_1-\theta_2)}+\cos{\varphi_1}\cos{\varphi_2}$. We can verify that the function $H(\varphi_1,\varphi_2,\theta_1,\theta_2)$ has a maximum value of 1. 

  For the second integrals, we can assume, without loss of generality, that $\sigma\subset[-4r_\sigma,0] \times [0, 4r_\sigma]\times [0, 4r_\sigma]$, $V_i\subset [0,4r_\sigma] \times [0, 4r_\sigma]\times [0, 4r_\sigma]$, and $\partial\sigma \cap \partial V_i\subset \{ 0\} \times [0,4r_\sigma]\times [0,4r_\sigma]$.  Then, we obtain that
  \begin{align*}
      &\quad\:\:\iint_{\sigma\times V_{i}}\frac{\text{d}x\text{d}y}{|x-y|^{3+2\epsilon}}\\
      & \leq\int_{-4r_{\sigma}}^{0}\int_{0}^{4r_{\sigma}}\int_{0}^{4r_{\sigma}}\int_{0}^{4r_{\sigma}}\int_{0}^{4r_{\sigma}}\int_{0}^{4r_{\sigma}}\frac{\text{d}y_{3}\text{d}x_{3}\text{d}y_{2}\text{d}x_{2}\text{d}y_{1}\text{d}x_{1}}{(|x_{1}-y_{1}|^{2}+|x_{2}-y_{2}|^{2}+|x_{3}-y_{3}|)^{\frac32+\epsilon}} 
      \\
      &\leq C r_{\sigma}^{3-2\epsilon}\int_{-1}^{0}\int_{0}^{1}\int_{0}^{1}\int_{0}^{1}\biggl[\int_{0}^{2}\int_{-1}^{1}\frac{2dzdw}{(|x_{1}-y_{1}|^{2}+|x_{2}-y_{2}|^{2}+|z|^{2})^{\frac32+\epsilon}}\biggr]\text{d}y_{2}\text{d}x_{2}\text{d}y_{1}\text{d}x_{1} 
      \\
      &\leq Cr_{\sigma}^{3-2\epsilon}\int_{-1}^{0}\int_{0}^{1}\int_{0}^{1}\int_{0}^{1}\biggl[\int_{-(|x_{1}-y_{1}|^2+|x_{2}-y_{2}|^2)^{-\frac12}}^{(|x_{1}-y_{1}|^2+|x_{2}-y_{2}|^2)^{-\frac12}}\frac{\text{d}\zeta}{(1+\zeta^{2})^{\frac32+\epsilon}}\biggr]\\
      &\qquad\cdot\frac{\text{d}y_{2}\text{d}x_{2}\text{d}y_{1}\text{d}x_{1}}{((x_{1}-y_{1})^2+(x_{2}-y_{2})^2)^{1+\epsilon}} 
      \\
      &\leq C_{\epsilon}r_{\sigma}^{3-2\epsilon}\int_{-\infty}^{\infty}\frac{\text{d}\zeta}{(1+\zeta^{2})^{\frac32+\epsilon}}\times\int_{-1}^{0}\int_{0}^{1}\int_{0}^{1}\int_{0}^{1}\frac{\text{d}y_{2}\text{d}x_{2}\text{d}y_{1}\text{d}x_{1}}{((x_{1}-y_{1})^2+(x_{2}-y_{2})^2)^{1+\epsilon}} 
      \\
      &\leq C_{\epsilon}r_{\sigma}^{3-2\epsilon}\int_{-\infty}^{\infty}\frac{\text{d}\zeta}{(1+\zeta^{2})^{1+\epsilon}}\times\int_{-1}^{0}\int_{0}^{1}\frac{\text{d}y_{1}\text{d}x_{1}}{(x_{1}-y_{1})^{1+2\epsilon}} 
      \\
      &\leq C_\epsilon r_\sigma^{3-2\epsilon}.
  \end{align*}
  Hence, we have verified \eqref{A.1_midres}.
\end{proof}

\bibliographystyle{model1-num-names}
\bibliography{biob.bib}

\end{document}